\documentclass[11pt,reqno]{amsart}

\usepackage[utf8]{inputenc}
\usepackage{graphicx}
\usepackage[dvipsnames]{xcolor}
\usepackage[colorlinks=true,linkcolor=BrickRed,citecolor=blue]{hyperref}
\usepackage{caption}
\usepackage{subcaption}
\usepackage{float}
\usepackage[width=150mm,top=35mm,bottom=25mm,bindingoffset=6mm]{geometry}
\usepackage{fancyhdr}
\usepackage{setspace}
\usepackage{amsmath}
\usepackage{amsthm}
\usepackage{amssymb}
\usepackage{amsfonts,mathrsfs}
\usepackage{dsfont}
\usepackage[shortlabels]{enumitem}
\usepackage{orcidlink}
\usepackage{mathtools}

\usepackage{tikz}
\usetikzlibrary{arrows.meta,positioning,calc,decorations.pathreplacing}

\newcommand{\vect}[1]{\boldsymbol{#1}} 

\renewcommand{\epsilon}{\varepsilon}
\newcommand{\eps}{\epsilon}

\newcommand{\E}{\mathbb{E}}
\newcommand{\Prob}{\mathbb{P}}
\renewcommand{\P}{\mathbb{P}}
\newcommand{\Z}{\mathbb{Z}}
\newcommand{\N}{\mathbb{N}}
\newcommand{\R}{\mathbb{R}}

\newcommand{\ceil}[1]{\lceil #1 \rceil}
\newcommand{\abs}[1]{|#1|}
\newcommand{\Or}{O}
\newcommand{\norm}[1]{\left\lVert#1\right\rVert}
\newcommand{\e}{\mathrm{e}}
\newcommand{\1}{\mathds{1}}
\newcommand{\dd}{\mathrm d} 
\newcommand{\B}{\mathfrak{B}}
\newcommand{\F}{\mathscr{F}}

\newtheorem{definition}{Definition}[section]

\newtheorem{thm}[definition]{Theorem}

\newtheorem{lem}[definition]{Lemma}
\newtheorem{prop}[definition]{Proposition}
\newtheorem{assump}[definition]{Assumption}

\theoremstyle{remark}
\newtheorem{rem}[definition]{Remark}

\numberwithin{equation}{section}  

\begin{document}

\title{Markov renewal theory for transfer operators and point processes on the line}
\date\today
\author{Yoon Jun Chan\orcidlink{0000-0002-7071-2885}}
\author{Markus Heydenreich\orcidlink{0000-0002-3749-7431}}
\address{Universität Augsburg, Department of Mathematics, 86135 Augsburg, Germany}
\email{yoon.chan@uni-a.de, markus.heydenreich@uni-a.de}

\author{Sabine Jansen\orcidlink{0000-0002-9611-0356}}
\address{Universität München, Mathematisches Institut, Theresienstr.\ 39, 80333 München, Germany}
\email{jansen@math.lmu.de}

\begin{abstract}
    We prove exponential decay of pair correlations for 1D stationary point processes when spacings satisfy a Markov condition, geometric ergodicity, and a condition on exponential moments. The conditions are phrased for stationary sequences of spacings (intervals between consecutive points) whose law comes from the Palm distribution of the point process. The key technical ingredient is a Markov renewal theorem with exponential convergence rate. The proofs combine classical regeneration techniques with the notion of geometric ergodicity for Markov chains with general state space. We apply the result to two models from statistical mechanics: (1) Gibbs point processes with a hard-core, finite-range pair potentials and (2) a harmonic chain of atoms, related to an autoregressive Gaussian process.
\end{abstract}


\date{July 27, 2026}

\maketitle

\vskip2em

\section{Introduction}\label{sec:intro}

Palm calculus is of fundamental importance in the theory of point processes. For point processes on the line, each point may represent an event -- e.g., the arrival of a new customer in queuing theory or the failure of a light bulb in renewal theory. Palm theory establishes a relation between stationary point processes $\Phi$ and stationary sequences $(Z_n)_{n\in \Z}$ of non-negative variables, roughly the waiting times between consecutive events for the point process conditioned on having a point at $t=0$. Independent waiting times give rise to renewal point processes. 

The present article is dedicated to a simple question: If the waiting times are not independent but the sequence of waiting times is stationary and has good mixing properties, does the point process inherit mixing properties of the same type? We investigate this question for exponential mixing and ask, more concretely: if 
\begin{equation}\label{eq:covspacdec}
	\Bigl| \mathrm{Cov}(Z_n, Z_{n+k})\Bigr| \leq C \e^{- ck }
\end{equation} 
for all $n\in \Z$ and uniform constants $c,C>0$, can we infer that the point process has exponential decay as 
\begin{equation}\label{eq:phidec}
	\Bigl| \mathrm{Cov}(\Phi(A),\Phi(B+t)) \Bigr| \leq C' \e^{- c't} 
\end{equation} 
for bounded $A,B\subset \R$? Put differently, we ask for a quantitative version of the cross-ergodic theorem from Palm theory \cite[Section 7.6]{Bremaud}. The latter states that the sequence $(Z_n)_{n\in \Z}$ is ergodic if and only if $\Phi$ is, but does not say anything about rates of mixing. 

Our principal motivation comes from statistical mechanics for particles on the line. Points of the point process represent particle locations and the $Z_n$'s represent spacings between nearest neighbours. There is a treasure-trove of techniques to prove exponential decay of correlations for sequences with Gibbs distribution but the situation for Gibbs point processes is less satisfying. A thorough understanding of how to go from decay of correlations for sequences to decay of correlations for point processes may help, in the future, to adapt techniques from one setting to the other.

Renewal theory indicates that the answer to our question is in general negative: In the simplest case of independent spacings, exponential decay occurs if and only if spacings have finite exponential moments. In discrete time, this is the content of Kendall's renewal theorem \cite[Theorem 15.1.1]{MeynTweedie}; for continuous time, see Nummelin and Tuominen \cite[Theorem 4.3]{NummelinTuominen} (that finite exponential moments are sufficient is already part of Stone's theorem \cite{Stone1966}). It is therefore clear that we have to impose additional conditions on the spacings. 

We answer the question formulated above under more restrictive conditions on $(Z_n)_{n\in \Z}$. Crucially we impose that the sequence is a $k$-step Markov chain for some finite $k\in \N$, i.e., the sequence of blocks $(Z_{n-k+1},\ldots, Z_n)$, $n\in \Z$, is Markov.  We allow for negative $Z_n$'s in order to cover the nearest-neighbour harmonic chain of atoms from statistical mechanics (see Section \ref{sec:harmonic-chain}). 
Our main result, Theorem \ref{thm:NewMain}, is succinctly summarised as follows: If the $k$-step Markov chain is geometrically ergodic and if the spacings satisfy $\limsup_{n\to \infty} \frac 1 n \log \E[\exp( t(Z_1+\cdots + Z_n))]<\infty$ for all sufficiently small $t$, then indeed~\eqref{eq:phidec} holds true. Geometric ergodicity strengthens condition~\eqref{eq:covspacdec} and the condition on exponential moments naturally generalises the condition $\E[\exp( t Z_1)]<\infty$ for i.i.d.\ spacings. (The theorem needs additional technical conditions, notably a kind of absolute continuity condition, known in renewal theory as \emph{spread out}, see e.g.\ \cite{AlsmeyerHoefsMRT}.) Thus, we identify sufficient conditions under which the answer to our question is positive.

The key technical ingredient is a Markov renewal theorem with exponential convergence rate (Theorem \ref{thm:renewal}). 
Markov renewal theorem is the analogue of the renewal theorem for a Markov renewal process, where the increments are no longer i.i.d.\ but modulated by a Markov chain.
Exponential convergence was established by Alsemeyer and Hoefs \cite{AlsmeyerHoefsMRT} in the case where the spacings form $m$-\emph{block factors} for some $m\in\N$.
Their main tool was the \emph{splitting (or regeneration) technique}, which was discovered independently by Nummelin \cite{Nummelin78} and Athreya and Ney \cite{AthreyaNey78}.
It constructs an embedded renewal process for the Markov chain and reduces the analysis to the embedded renewal process, and admits lifts of established results for Markov chains on a countable state space to Markov chains on uncountable state spaces.
The main challenge of proving Theorem \ref{thm:renewal} via the splitting technique is to identify moment conditions on the increments $(Z_n)_n$ \cite{AlsmeyerHoefsMRT}.
We were able to translate the moment condition for a renewal process \cite{Stone1966} to a suitable moment condition (Assumption \ref{MAssu:ExpMoment}) for a Markov renewal process by combining the \emph{cyclic decomposition} with Lemma \ref{lem:tau_geom_tail}, which states that the regeneration introduced by the splitting technique happens exponentially fast.
Fuh \cite{Fuh} proved exponential convergence for a more general class of models via an analytical approach.

We apply our main result to two classes of models and obtain correlation decay in the sense of \eqref{eq:phidec}.
The first is one-dimensional Gibbs point processes with a finite-range pair potential with hard core.
Gibbs point processes constitute a central object in modern statistical mechanics and remain an active area of research (e.g., \cite{Dereudre2019,Georgii+2011,Ruelle-book} and references therein).
The second is harmonic chains of atoms, often studied in elasticity theory, see e.g.\ \cite{AllmanBetzHairer, AurzadaBetzLifshits}.
These are a finite chains with one end fixed, where atoms interact with their neighbours and next-nearest neighbours via a quadratic potential.
They also form a first-order autoregressive process with Gaussian noise.
The Palm inversion formula for one-dimensional systems is usually formulated in the literature only for the case where $Z_n>0$ almost surely.
The study of this autoregressive process demonstrates that our method is not restricted to such cases.
A similar approach has been carried out by Athreya, Tweedie and Vere-Jones \cite{ATVJ} to prove the Markov renewal theory without convergence rate for the Wold process, where the interparticle distances form a Markov chain, and later extended by Chong \cite{FSChong} to the case of $2$-step Markov chains.

The article is organised as follows.
In Section \ref{sec:results} we explain the set-up of our model and state our main result and we apply it to the two examples mentioned above in Section \ref{sec:StatMech}.
We summarise the splitting technique and discuss the implications of geometric ergodicity in Section \ref{sec:splitting}.
Utilising the embedded renewal process, we prove our first main result in Section \ref{sec:ProofNewMain}.
We explain in detail the applicability of our main results on the two examples in Section \ref{sec:ProofStatMech} and Section \ref{sec:AR1}.

\section{Main result}\label{sec:results}
Let $(Z_n)_{n\in \Z}$ be a stationary and ergodic sequence of real-valued random variables. 
``Stationary'' means $(Z_{n-1})_{n\in \Z}$ is equal in distribution to $(Z_n)_{n\in \Z}$, and ``ergodic'' means that the $\sigma$-algebra of events invariant under shift of indices is trivial.
We assume throughout that $\P(Z_i =0) = 0$  and that the $Z_i$'s have a finite, strictly positive expected value. Let $(X_n)_{n\in \Z}$ be the sequence defined by 
\begin{equation}
    \label{eq:defX}
    X_{0} = 0, \  X_n - X_{n-1} = Z_{n}, \quad n\in \mathbb{Z}.
\end{equation}
We would like to describe the configuration of particles whose positions are $(X_n)$ as a point process in the sense of \cite[Chapter 9]{DaleyJones032}.
We recall the following result on the existence and uniqueness of the translation invariant point process constructed from $(Z_n)_{n\in \Z}$.
\begin{lem} 
	There exists a stationary point process $\Phi$ with intensity $1/\E[Z_1]$ and Palm version 
	\[
		\Phi^0 = \sum_{n\in \Z}\delta_{X_n}. 
	\]
	The law of $\Phi$ is uniquely determined.
\end{lem} 

\begin{proof}
	For non-negative spacings $(Z_n)$, this follows directly from standard Palm theory, see e.g.\ \cite[Theorem 13.3.I]{DaleyJones032} and \cite[Corollary 6.2.12]{baccelli}.
	For the general case, Berbee studied exactly the random walk with stationary increments \cite[Section 2.1]{Berbee}.
	In particular, he showed that the one-to-one correspondence between the stationary sequence of non-negative increments and a stationary point process extends to a stationary and ergodic sequence of increments.
\end{proof}

We now state five assumptions,  which we need to formulate our main result.
\begin{assump}[$k$-step Markov] \label{ass:markov}
	The process $(Z_n)_{n\in \Z}$ is $k$-step Markov, i.e., the chain 
        \begin{equation}
            \label{eq:defM}
            M_n = (Z_{n-k+2},\ldots, Z_{n+1})
        \end{equation}
	is a Markov chain for some transition kernel $P\colon\R^k\times \B_{\R^k}\to [0,1]$ with Borel $\sigma$-algebra $\B_{\R^k}$.
\end{assump}

\begin{figure}
    \centering
    \begin{tikzpicture}[x=0.8\linewidth,y=5cm,>=Stealth]

        \draw[->] (0.1,0.7) -- (0.9,0.7);
    
        \foreach \PP in {-3,-2,-1,1,2,3}
            \draw (0.5+\PP*0.1,0.72) -- (0.5+0.1*\PP,0.68) node[below=2pt] {$X_{\PP}$};
    
        \draw (0.5,0.72) -- (0.5,0.68) node[below=2pt] {$X_{0}=0$};
    
        \foreach \PP in {-2,-1,0,1,2,3}
            \draw[->] (0.405+\PP*0.1,0.77) -- (0.495+\PP*0.1,0.77) node[midway,above] {$Z_{\PP}$};

        \node at (0.15,0.77) {$\cdots$};
        \node at (0.85,0.77) {$\cdots$};

        \draw[decorate, decoration={brace,amplitude=10pt}] (0.2,1.05)--(0.49,1.05) node[right,above=10pt] {$M_{-1}=(Z_{-2},Z_{-1},Z_0)$};
        \draw[decorate, decoration={brace,amplitude=10pt}] (0.3,0.85)--(0.59,0.85) node[above=10pt] {$M_{0}=(Z_{-1},Z_0,Z_1)$};
        
    \end{tikzpicture}
    \caption{An example configuration of the Markov random walk defined by \eqref{ass:markov} and \eqref{eq:defX} with $k=3$. Notice that $M_i$ and $M_{i+1}$ have overlapping components.}
    \label{fig:inversion}
\end{figure}
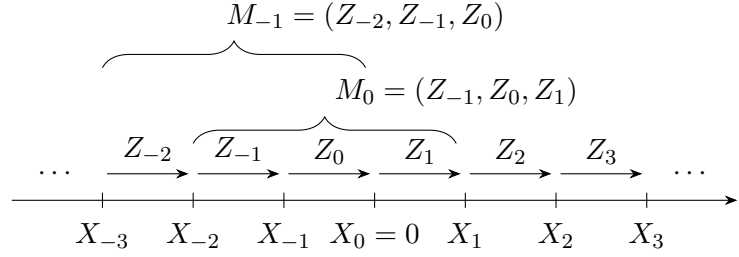

The index choice in \eqref{eq:defM} is explained in Remark \ref{rem:indexM}. 
Let $\pi$ be the distribution of $M_0$. As $(Z_n)_{n\in \Z}$ is stationary, $(M_n)_{n\in \Z}$ is stationary as well and $\pi$ is an invariant measure for the chain, $\pi P = \pi$. 

\begin{assump}[Geometric ergodicity] \label{MAssu:RegularityMC}
	The chain $(M_n)_{n\in\N_0}$ is \emph{geometrically ergodic}, i.e., 
	there exists $q\in (0,1)$ and, for every $\vect z \in \R^k$, a constant $C(\vect z)<\infty$ such that 
	\[
		|| P^n(\vect z, \cdot) - \pi ||_\mathsf{TV} \leq C(\vect z) \, q^n
	\] 
	for all $\vect z\in \R^k$ and all $n\in \N$. 
\end{assump} 

The next definition is borrowed from Markov renewal theory. Allowing for initial laws for $M_0$ different from $\pi$ and focusing on $n\geq 0$, the chain $(M_n,X_{n})_{n\in \N_0}$ (with $X_{0}=0$) is an example of a  \emph{Markov random walk}, see e.g.\ \cite{Alsmeyer97} for a precise definition. The Markov random walk has \emph{positive drift} if  $\E[Z_i]>0$ . For i.i.d.\ spacings $(Z_n)$, the chain $(X_{n})_{n\in \N_0}$ is a standard random walk.
We say that the Markov random walk is spread out if there exists $n_0\in \N$ such that $X_{n_0}$ has an absolutely continuous component with respect to Lebesgue measure. For Markov random walks, a similar condition is imposed on $(M_n, X_{n})_{n\in \N_0}$, taking into account the driving chain. See Alsmeyer and Hoefs \cite{AlsmeyerHoefsMRT}, and Niemi and Nummelin \cite{NiemiNummelin}. 
\begin{assump}[Spread out] \label{Massu:spread_out}
    The Markov random walk $(M_n,X_{n})_{n\in \N_0}$ is spread out, i.e., 
	there exists $n_0\in \N$ such that for all initial values $\vect z$ of $M_0$ the distribution of $(M_{n_0},X_{n_0})$  under $\P_{\vect z}$ ($(M_n)_{n\geq 0}$ started in $\vect z$) is non-singular with respect to $\pi\otimes \mathrm{Leb}$, where $\mathrm{Leb}$ is Lebesgue measure on $\R$. 
\end{assump} 
A sufficient condition is that $P(\vect z, \cdot)$ is absolutely continuous with respect to Lebesgue measure on $\R^k$. This simpler sufficient condition is satisfied in all of our concrete examples. 
We also need that the growth rate of the exponential moments of $Z_0+\cdots+Z_n$ is finite in $n$.
\begin{assump}[Exponential moments] \label{MAssu:ExpMoment} 
	There exists $\eps>0$ such that for all $t\in (-\eps,\eps)$, the stationary sequence $(Z_n)_{n\in \Z}$ satisfies 
	\begin{equation} \label{eq:kMGF}
		\limsup_{n\to \infty}\frac 1n \log \E\bigl[\e^{t (Z_1+\cdots + Z_n)}\bigr] <\infty. 
	\end{equation}  	
\end{assump} 

Finally, in the case where $Z_n$ might be negative, we assume reversibility.
\begin{assump}[Positive or reversible] \label{MAssu:Inversion}
	The stationary sequence $(Z_n)$ satisfies
	\begin{itemize} 
		\item $Z_n>0$ almost surely, or 
		\item the sequence is reversible, i.e., $ (Z_{-n})_{n\in \Z}\stackrel{d}{=} (Z_n)_{n\in \Z}$.
	\end{itemize} 
\end{assump} 

\begin{rem} \label{rem:indexM}
    The shift in the index is motivated by the application of renewal theory.
    The index choice is $(Z_{n-k+2},\ldots, Z_{n+1})$ is such that at regeneration times $\tau_j$, to be defined in Section \ref{sec:splitting}, the two components of $(M_{\tau_j},X_{\tau_j})$ are independent to each other and thus $(X_{\tau_j} - X_{\tau_{j-1}})_{j\in\N}$ forms an i.i.d.\ sequence.
\end{rem}

Our main result establishes exponential decay of correlations at the level of second-order moment measures. 
\begin{thm}[Exponential decay of correlations]
    \label{thm:NewMain}
    Let $(Z_n)_{n\in\mathbb{Z}}$ be a stationary process sequence that satisfies Assumptions \ref{ass:markov}-\ref{MAssu:Inversion} and let  $(X_n)_{n\in \Z}$ be defined as in~\eqref{eq:defX}.  Then the point process $\Phi$ with Palm version $\Phi^0=\sum_{n\in\mathbb{Z}} \delta_{X_n}$ satisfies the following:  there exists $\epsilon >0$ such that, for all bounded intervals $A,B \subset \R$,
    \begin{equation*}
		\E\bigl[ \Phi(A) \Phi(B+t)\bigr] - \E\bigl[ \Phi(A) \bigr]\E\bigl[ \Phi(B) \bigr] = \Or(\e^{-\epsilon t})  \quad \textrm{as } t\to\infty.
    \end{equation*}
\end{thm}

The study of decay of correlation using discrete renewal theory was done by Giacomin \cite{Giacomin} for homogeneous pinning models, whereas we apply the continuous renewal theory to study systems in continuum.
Mixing properties of point processes have been studied under various notions of mixing; we refer to \cite{Bradley} for an introduction.
Although decay of the second-order moment measure is in general weaker than the widely studied $\alpha$-mixing, it can be useful in certain cases.
For example, Poinas, Delyon and Lavancier \cite{PoinasDelyonLavacier} gave a bound on the $\alpha$-mixing coefficient for associated processes using the truncated correlation function.
For Gibbs point processes in the Dobrushin uniqueness region, $\alpha$-mixing can be deduced from a bound on the covariance, see Jensen \cite{Jensen}.

\section{Application to one-dimensional Gibbs measures}\label{sec:StatMech}
We discuss two examples of point processes to which we apply Theorem \ref{thm:NewMain}: Gibbs point processes on the line for finite-range pair potentials with a hard core and one-dimensional harmonic chains of atoms with next-nearest neighbour interaction.

\subsection{Gibbs point processes on the line} 
Let $\mathbf{N}$ be the space of locally finite counting measures on $\R$,
equipped with the $\sigma$-algebra generated by the sets $\{\eta \in\mathbf N\colon \eta(B) = k\}$, $B\in \B_\R$, $k\in \N_0$, see \cite[Chapter 9]{DaleyJones032}.
A point process is an $\mathbf{N}$-valued random variable.
Let $v\colon\R_+\to \R\cup \{\infty\}$ be a measurable function that is bounded from below. 
A point process $\Phi$ on the line is a Gibbs point process at activity $\zeta>0$ and inverse temperature $\beta >0$ for the pair potential $v$ if, for every measurable test function $F\colon\R\times \mathbf N \to \R_+$, 
\begin{equation}
	\label{eq:GNZ}
	\E\Bigl[ \int_\R F(x,\Phi) \Phi(\dd x) \Bigr] = \int_\R \E\Bigl[ \zeta\, \e^{ - \beta W(x,\Phi)} F(x,\Phi)\Bigr] \dd x,
\end{equation}
where 
\[
	W(x,\Phi) = \begin{cases}
		\int_\R v(|x-y|) \Phi(\dd y) &\quad\text{if } \int_\R |v(|x-y|)|\Phi(\dd y) <\infty,\\
		\infty&\quad \text{else}. 
	\end{cases} 
\] 
Equation~\eqref{eq:GNZ} is called the GNZ equation, and it is equivalent to the DLR equation, see \cite{NZ_1979}.

A probability measure $\P$ on $\mathbf N$ is a Gibbs measure if it is the law of a Gibbs point process. The set of Gibbs measures at $\beta,\zeta$ is denoted $\mathscr G(\beta,\zeta)$. It is well known that if $v$ is superstable \cite{Ruelle_Superstable} and $\int_R^\infty r |v(r)| \dd r <\infty$ for some $R>0$, then for every $\beta,\zeta>0$, the set $\mathscr G(\beta,\zeta)$ consists of a single measure $\mathsf P = \mathsf P_{\beta,\zeta}$ that is translationally invariant and ergodic \cite{Papangelou_1987}. This is the case, for example, when $v$ has a finite range $R<\infty$ (it vanishes on $[R,\infty)$) and a hard core $r_{\mathrm{hc}}>0$ ($v\equiv \infty$ on $[0,r_\mathrm{hc}]$). 
 
Let $\Phi$ be a Gibbs point process i.e.\ $\Phi\sim \mathsf P_{\beta,\zeta}$ and the spacings $(Z_n)_{n\in \Z}$ form a stationary sequence so that $\sum_{n\in \Z} \delta_{X_n}$ with $X_0=0,\,X_n - X_{n-1} = Z_n$  is a Palm version of $\Phi$. 

\begin{thm} \label{thm:gibbs-markov}
	Assume that the pair potential $v$ has a hard core $r_\mathrm{hc} >0$, finite range $R<\infty$, and is finite and bounded from below on $[r_\mathrm{hc}, R]$.  Let $k\in \N$ with $(k+2)r_\mathrm{hc} \geq R$. Let $(Z_n)_{n\in \Z}$ be a stationary sequence of spacings associated with a Gibbs point process $\Phi \sim\mathsf P_{\beta,\zeta}$. Then: 
    \begin{enumerate}
    \item [(a)]
	$\bigl((Z_{kn-k+1},\ldots, Z_{kn})\bigr)_{n\in\Z}$, is a geometrically ergodic stationary Markov chain with state space $(r_\mathrm{hc},\infty)^{k}$. 
    \item [(b)] For all bounded intervals $A,B\subset \R$ and some $\eps>0$
	\begin{equation*}
		\E\bigl[ \Phi(A) \Phi(B+t)\bigr] - \E\bigl[ \Phi(A) \bigr]\E\bigl[ \Phi(B) \bigr] = \Or(\e^{-\epsilon t}) \quad \textrm{as } t\to\infty.
    \end{equation*}
    \end{enumerate}
\end{thm} 

Exponential decay of correlations for Gibbs point processes with finite-range potentials that have a hard core is known. Although no explicitly stated so, it should readily follow, for example, from the classical work of Gallavotti and Miracle-Sole \cite{GallavottiMiracleSole} using transfer operators. 
There are various other methods to show exponential decay of correlations for both lattice and continuum systems, for example, the Witten Laplacian method by Helffer, Sj\"ostrand and Bach, M{\o}ller \cite{BachMoller,HelfferSjostrand}, cluster expansion by Procacci, Scoppola and Ueltschi \cite{procacci2001decay,Ueltschi} and the Dobrushin uniqueness technique by Gross, Heinrich and Klein \cite{gross1979decay,Heinrich,klein1982dobrushin}.

Our contribution is a novel purely probabilistic proof that clarifies the relation with Palm theory and renewal theory and is of relevance in the theory of point processes. In addition, we show the Markovian nature of the spacings explicitly; this is of no surprise in mathematical physics, but it is, to the best of our knowledge, not proven explicitly in the literature. The proof is of interest in its own right as it exploits the notion of canonical Gibbs measure and the equivalence of ensembles proven by Georgii \cite{georgii1976canonical,georgii-canonical-book}. 

\medskip
Let us briefly sketch how to obtain the transition kernel of the Markov chain in Theorem \ref{thm:gibbs-markov} for the simplest case $3 r_\mathrm{hc} \geq R$, i.e.\ $m=2$. Given $\beta,p>0$, define
$K_{\beta,p}\colon \R_+\times \R_+ \to \R_+$ by 
\[
	K_{\beta,p}(z_1,z_2) = \exp\Bigl( - \beta \Bigl[ \frac 1 2 (pz_1+ v(z_1)) + v(z_1+z_2) + \frac 12 (p z_2 + v(z_2))\Bigr]\Bigr). 
\]  
The integral operator with kernel $K_{\beta,p}$ is the \emph{transfer operator}, a useful tool for studying statistical mechanics models, whose eigenvalues can be used to show correlation decay and absence of phase transitions, see Chapter 5.6 in Ruelle \cite{Ruelle-book}. 
More importantly for us, it provides a way to connect one-dimensional point processes with Markov chains such that the law of the sequence of spacing and the law of the Markov chain coincide. 

In Section~\ref{sec:transfer} we shall see that, by  the Krein-Rutman theorem \cite[Theorem 19.3]{Deimling}, there exist $\lambda(\beta,p)>0$ and a non-negative function $\varphi_{\beta,p}\colon\R_+\to \R_+$ that vanishes on $[0, r_\mathrm{hc}]$ and is strictly positive on $[r_\mathrm{hc},\infty)$ so that 
\[
	\int_{\R_+} K_{\beta,p}(z_1,z_2) \varphi_{\beta,p}(z_2) \dd z_2 = \lambda(\beta,p) \varphi_{\beta,p}(z_1)
\] 
for all $z_1\geq 0$. Define 
\begin{equation}\label{eq:Transfer_to_transition}
	P_{\beta,p}(z,A):= \frac 1 {\lambda(\beta,p) \varphi_{\beta,p}(z)} \int_A K_{\beta,p}(z,z') \varphi_{\beta,p}(z') \dd z'.    
\end{equation}
Then given $\beta,\zeta\geq 0$ there exists $p=p(\beta,\zeta)$ so that $P_{\beta,p}(x,\dd y)$ 
is a transition kernel for the chain $(Z_n)_{n\in \Z}$ in Theorem \ref{thm:gibbs-markov} (for $m=2$). 

\subsection{Harmonic chain of atoms} \label{sec:harmonic-chain}
Our next example is a harmonic chain of atoms with next-nearest neighbour interaction in equilibrium, also called a harmonic crystal and closely related to the discrete Gaussian free field, see Chapter 8 in Friedli and Velenik \cite{FV}.
These are systems of particles on the real line coupled via a quadratic potential, as if they are attached to their neighbours (and next-nearest neighbours) with a spring, with one end fixed. These models, with possibly more general interaction, are invariant measures for dynamics driven by noise, are used to model rupture of materials \cite{AllmanBetzHairer,AurzadaBetzLifshits}.
Let $k_1,k_2>0$ and $a>0$. 
A configuration with $n+1$ atoms with locations $x_1,\ldots, x_n$ and one particle pinned at a deterministic position $x_0$ (e.g.\ $x_0=0$) has energy 
\begin{equation}\label{eq:eharmonic}
	U_{n}^{x_0}(x_1,\ldots, x_n ) = \sum_{i=1}^n \frac 12 k_1 (x_{i} - x_{i-1} - a)^2 + \sum_{i=2}^n \frac 12 k_2 (x_{i} - x_{i-2} - 2 a)^2.
\end{equation}
The associated Gibbs measure at inverse temperature $\beta >0$ is the  probability measure $\mu_{n,\beta}^{x_0}$ on $\R^n$ with probability density function proportional to $\exp( - \beta U_n)$. The measure favours ordered $x_i$'s with $z_i = x_{i} - x_{i-1}\simeq a >0$ but it allows for negative increments $z_i  <0$. 
Let 
\begin{equation}\label{eq:cgammadef}
	c:= \sqrt{\Bigl( \frac {k_1}{2}+k_2 \Bigr)^2 - k_2^2},\quad \gamma:= \frac{k_1}{2}+k_2 + c
\end{equation}
and 
\begin{equation}\label{eq:pharmonic}
	P_\beta (z,B):= \sqrt{\frac{\beta \gamma}{2\pi}}\, \int_B \exp\Bigl( - \frac 12 \beta \gamma \Bigl(z'-a + \frac{k_2}{\gamma}(z-a)\Bigr)^2  \Bigr)  \dd z'.
\end{equation} 
The kernel $P_\beta$ is the transition kernel for a Markov chain, and the normal law $\mathcal N(a, 1/(2 \beta c))$ is reversible. Let $(Z_n)_{n\in \Z}$ be the associated stationary Markov chain and $\Phi$ the translationally invariant point process with Palm version $\Phi^0 = \sum_{n\in \Z} \delta_{X_n}$ where $X_n- X_{n-1} =Z_n$ and $X_0=0$. 

The next theorem identifies the law of $\Phi$ as the limit of the Gibbs measures $\mu_{n,\beta}^{x_0}$ when we let $n\to \infty$ first and then $x_0 \to -\infty$. Thus we may think of $\Phi$ as the locations of the atoms in an infinite harmonic chain.
A function $f$ on $\mathbf{N}$ is called local if there exists a bounded measurable region $\Lambda \subset \R$ such that for all $\eta,\gamma \in \mathbf{N}$ with $\eta(\cdot \cap \Lambda) = \gamma(\cdot \cap \Lambda)$ we have $f(\eta)=f(\gamma)$.

\begin{thm} \label{thm:harmonic}
    Let $\mu_{n,\beta}^x, \Phi$ be defined as above.
    \begin{enumerate}
        \item [(a)] For every local bounded measurable function $F\colon\mathbf N\to \R$, 
            	\[
            	   \lim_{x_0\to -\infty} \lim_{n\to \infty} \int_{\R^n} F(\delta_{x_1}+\cdots +\delta_{x_n}\bigr) \dd \mu_{n,\beta}^{x_0}(x_1,\ldots,x_n) = \E\bigl[F(\Phi)]. 
                \] 
        \item [(b)] For all bounded intervals $A,B\subset \R$ and some $\eps>0$,
        	\begin{equation*}
        		\E\bigl[ \Phi(A) \Phi(B+t)\bigr] - \E\bigl[ \Phi(A) \bigr]\E\bigl[ \Phi(B) \bigr] = \Or(\e^{-\epsilon t})  \quad \textrm{as } t\to\infty.
            \end{equation*}
    \end{enumerate}
\end{thm}

\begin{rem}
    \label{rem:AR1}
	The chain $(Z_n-a)_{n\in \N}$ is an autoregressive process of order 1 with Gaussian noise: Indeed, let $(\xi_n)_{n\in\Z}$, be i.i.d.\ standard normal random variables. Then, given $Z_0$, the chain $(Z_n)_{n\geq 0}$ defined by 
	\begin{equation}\label{eq:AR1}
		Z_{n+1} -a = - \frac{k_2}{\gamma} (Z_n -a)  + \frac 1{\sqrt{\beta \gamma}}\, \xi_n.  
	\end{equation} 
	is a Markov chain with transition kernel~\eqref{eq:pharmonic}.	Autoregressive processes are an important building block for more complicated models used to study time series, e.g.\ in signal processing and mathematical finance, see \cite{Brockwell_Davis}.
\end{rem} 

\section{The splitting technique and regeneration times} \label{sec:splitting}

For the readers' convenience, we summarise here the splitting technique, developed independently by Nummelin \cite{Nummelin78} and Athreya and Ney \cite{AthreyaNey78}.
This technique also allows us to define the regeneration epoch $\tau$, and show how geometric ergodicity implies a geometric tail of $\tau$.
We follow Meyn and Tweedie  \cite[Chapter 5]{MeynTweedie} closely. Sections~\ref{sec:splitM} and~\ref{sec:tau} work under the additional assumption that the Markov chain is strongly aperiodic. In Section~\ref{sec:splitMX} we drop the assumption and apply the splitting construction to the bivariate chain $(M_{rn}, X_{rn} - X_{rn-r})$ with suitable $r$.

\subsection{The splitting technique} \label{sec:splitM} 

We illustrate the splitting technique by proving Lemma~\ref{lem:splitM} below, a classical result. This lemma, together with Lemma~\ref{lem:tau_geom_tail} on the tails of $\tau$, is sufficient for the proofs of Theorem~\ref{thm:gibbs-markov} in the case $m=2$ and Theorem \ref{thm:harmonic}. 

By Assumption \ref{MAssu:RegularityMC} on geometric ergodicity,  the driving chain $(M_n)_{n\in \N_0}$ is \emph{$\pi$-irreducible} \cite[Chapter 4.2]{MeynTweedie}, i.e., every set of positive $\pi$-measure can be reached from any initial configuration: 
\begin{equation} \label{eq:irreducibility}
	\pi(A) > 0\ \Rightarrow\ \forall x \in \R^k\ \exists \ell \in \N\ P^\ell(x,A) >0. 
\end{equation}
In addition, ergodicity implies that the chain is Harris recurrent, i.e.\ 
\[
	\pi(A) >0\ \Rightarrow\ \forall x \in A\colon\ \P_x\bigl( (M_n)\text{ visits }A\text{ infinitely often}\bigr) = 1,
\] 
see e.g.~\cite[Chapter 9]{MeynTweedie} and \cite[Proposition 6.3]{Nummelin_1984_book}. As it has an invariant probability measure, it is a \emph{positive Harris chain} \cite[Proposition 6.3]{Nummelin_1984_book}. Because of $\pi$-irreducibility \eqref{eq:irreducibility}, the minorisation lemma \cite[Theorem 5.2.3]{MeynTweedie} applies: there is a set $R$ with $\pi(R)>0$, an integer $r\in \N$, and a $\lambda \in (0,1)$ such that the \emph{minorisation condition} holds, i.e.\ 
\begin{equation} \label{eq:minor}
	P^r(x, A) \geq \lambda \nu(A),\quad \nu(A):= \frac{\pi(A\cap R)}{\pi (R)}
\end{equation} 
for all $A$, see \cite[Chapter 5.2]{MeynTweedie}. 

Suppose that we may take $r=1$ in~\eqref{eq:minor}, i.e., the chain is \emph{strongly aperiodic} in the sense of Athreya and Ney \cite{AthreyaNey78}.
For $x\in R$ and measurable $A\subset \R^k$, define $Q(x,A)$ by 
\[
	P(x,A) = \lambda \nu(A) + (1-\lambda) Q(x,A),
\] 
where $\lambda,\nu$ are taken from the minorisation condition \eqref{eq:minor}.
The minorisation condition \eqref{eq:minor} for $r=1$ guarantees that $Q(x,\cdot)$ is a probability measure. One may now slightly modify the Markov dynamics and define a new chain. The new chain evolves as $(M_n)$ until it hits the \emph{regeneration set} $R$ in some point $x\in R$. When it does, a coin is flipped; with probability $\lambda$, the next location is chosen according to $\nu$, with probability $1-\lambda$, it is chosen according to $Q(x,\cdot)$. This is essentially the construction in Athreya and Ney \cite{AthreyaNey78}. 

Another, closely related, approach is the split chain, a Markov chain $(M_n^*)$ with state space $\R^k\times \{0,1\}$ and marginal $(M_n)$, see Nummelin \cite{Nummelin78} and Meyn and Tweedie \cite[Chapter 5.1]{MeynTweedie}. Every probability measure $\mu$ on $\R^k$ is lifted to a measure $\mu^\ast$ on $\R^k\times \{0,1\}$ by 
\begin{equation*}
	\mu^\ast(A\times \{0\}) = (1-\lambda) \mu(A\cap R) + \mu(A\cap R^\mathrm c),\quad 
   \mu^\ast(A\times \{1\}) = \lambda \mu (A\cap R).
\end{equation*}
Let us write $\mathcal M_1(\R^k)$ for the set of probability measures on $\R^k$. 

\begin{lem} \label{lem:splitM}
	Let $P$ be the transition kernel of a positive Harris chain with state space $\R^k$ and invariant probability measure $\pi$. Suppose that the minorisation condition~\eqref{eq:minor} holds with $r=1$. Then there exists a tuple 
	\[
		\Bigl( \Omega,\F, (\P_{\mu^\ast})_{\mu \in \mathcal M_1(\R^k)}, (M_n^\ast)_{n\in \N_0}, (\tilde \tau_j)_{j\in \N_0} \Bigr) 
	\] 	
	in which $(\Omega,\F)$ is a measurable space, each $\P_{\mu^\ast}$ is a probability measure on $(\Omega,\F)$, and:
	\begin{enumerate} [(i)]
		\item Under $\P_{\mu^\ast}$, $(M_n^\ast)_{n\in \N_0}$ is a time-homogeneous Markov chain with initial law $\mu^\ast$.
		\item $0< \tilde \tau_1< \cdots $ is a sequence of  $\P_{\mu^\ast}$-almost surely finite stopping times  (with respect to the natural filtration $(\F_n)_{n\in\N_0}$ of $(M_n^\ast)_{n\in \N_0}$).
		\item For all measurable sets $A\subset \R^k$, and every probabilty measure $\mu$ on $\R^k$,
                \begin{equation} \label{eq:regeneration_M}
            		\P_{\mu^\ast}\Bigl( M_{\Tilde{\tau}_j}^\ast  \in A\times \{0,1\} \, \Big|\, \F_{\tilde \tau_j -1} \Bigr) = \nu(A),
            	\end{equation}
            	$\P_{\mu^\ast}$-almost surely.
        \item   The driving chain $(M_n)_{n\in \N_0}$ with initial law $\mu$ is the marginal of $(M^\ast_n)$ with initial law $\mu^\ast$, i.e.\ for every probability measure $\mu$ on $\R^k$ and $n\in\N$ and all measurable $A \in \R^k$,
            	\begin{equation} \label{eq:split_marginal}
 			\P_{\mu^\ast}(M^\ast_n \in A \times \{ 0,1 \}) = P^n(\mu,A).
            	\end{equation}
	\end{enumerate} 
\end{lem}

\begin{proof}
	The lemma is well known, we give the proof because we use details of the construction later on. Define a Markov kernel $\hat P$ on $\R^k\times\{0,1\}$ by 
	\begin{align*}
		    \hat{P} \bigl( (x,0), \cdot\bigr) &= P (x, \cdot)^\ast, & x \in R^\mathrm c, \\
		    \hat{P} \bigl( (x,0), \cdot\bigr) &= Q(x,\cdot)^\ast, & x \in R, \\
		    \hat{P} \bigl( (x,1),\cdot \bigr) &= \nu^\ast, &  x \in \R^k. 
	\end{align*}
	Clearly there is a family  $(\Omega,\F,(\P_{\mu^\ast})_{\mu \in \mathcal M_1(\R^k)}, (M_n^\ast)_{n\in \N_0})$ in which each $\P_{\mu^\ast}$ is a probability measure and under $\P_{\mu^\ast}$, is a Markov chain with transition kernel $\hat P$. The kernel satisfies 
	\[
		\int \mu(\dd x) P^\ell(x,A) = \int \mu^\ast (\dd x^\ast) {\hat P}^\ell(x^\ast, A\times\{0,1\})
	\]  
	for all Borel sets $A\subset \R^k$, all $\ell \in \N$ and every probability measure $\mu$ on $\R^k$ \cite[Theorem 5.1.3]{MeynTweedie}.  This proves items (i) and (iv). 
	Harris recurrence implies that the hitting time 
	\begin{equation} \label{eq:def_tau-1}
    		\sigma_1:= \inf \{n\in \N_0\colon\ M_n^\ast \in \R^k\times\{1\}\}
	\end{equation}
	is almost surely finite, see \cite{AthreyaNey78}.
	We define $\Tilde{\tau}_1:= 1+ \sigma_1$ and then inductively a sequence of stopping times $(\Tilde{\tau}_j)$ (with respect to the natural filtration of $(M_n^\ast)$) by 
	\[
		\Tilde{\tau}_{j+1}:= 1+ \inf\bigl\{ n\geq \Tilde{\tau}_j\colon\, M_n^\ast \in \R^k\times\{1\}\bigr\}. 
	\] 
      Item (iv) follows from the strong Markov property and $\hat P( (x,1),A \times \{0,1\}) = \nu(A)$ for all $x\in \R^k$,
\end{proof} 

\subsection{Geometric tail of the regeneration time} \label{sec:tau} 
We now show that geometric ergodicity implies that the regeneration time $\tilde \tau_1$ has finite exponential moments. Remember the regeneration set $R$ in the minorisation condition \eqref{eq:minor}, the invariant measure $\pi$ and the measure $\nu(B) = \pi(B\cap R)/\pi(B)$. 

\begin{lem}\label{lem:tau_geom_tail}
	Let $P$ be the transition kernel of a geometrically ergodic (in the sense of Assumption~\ref{MAssu:RegularityMC}) Markov chain with state space $\R^k$. Suppose that $P$ satisfies the minorisation condition~\eqref{eq:minor} with $r=1$. Then, making the regeneration set $R$  in~\eqref{eq:minor} smaller if needed, we may assume that the stopping time $\tilde \tau_1$ in Lemma~\ref{lem:splitM} satisfies, for $\mu \in \{\pi,\nu\}$ and some $\eps>0$, 
	\begin{equation*}
		\E_{\mu^\ast} \bigl [\e^{\epsilon \tilde{\tau}_1}\bigr] < \infty.
	\end{equation*} 
\end{lem}

\begin{proof}
	First we check that, upon making the regeneration set smaller if needed, there exist constants $C>0$ and $ \delta \in (0,1)$ such that 
	\begin{equation}\label{eq:rgeom1}
		\sup_{x\in R} || P^n(x,\cdot) - \pi ||_\mathsf{TV}  \leq C \delta^ n
	\end{equation} 
	 for all $n\in \N$. To that aim, we adapt an argument in Roberts and Rosenthal~\cite{RobertRosenthal}. As $\R^k$ is countably generated, the map $x\mapsto ||P^n(x,\cdot) - \pi||_\mathsf{TV}$ is measurable, for every $n$ \cite[Appendix]{RobertRosenthal}.  Therefore, setting $\delta:= q$ with $q$ from Assumption~\ref{MAssu:RegularityMC},  $C(x):=\limsup_{n\to \infty}  \delta^{-n} ||P^n(x,\cdot)- \pi||_\mathsf{TV}$ is a  measurable function of $x$ as well. By the assumption on geometric ergodicity, $C(x)<\infty$ for all $x\in \R^k$. Define
\begin{equation*}
	B_d = \{ x\in R \mid C(x)\leq d \}.
\end{equation*}
	Since $C(x)<\infty$ on $\R^k$, we must have $\cup_{d \in \N} B_d = R$, hence 
 $\lim_{d\rightarrow \infty} \pi(B_d) = \pi(R) >0$ and 
there exists $D>0$ such that $\pi(B_D)>0$.  The set $R':= B_D$ satisfies the minorisation condition with $r=1$ as well:  for all $E\subset R$ and all $x\in R$, 
	\[
		P(x,E) \geq \lambda \nu (E \cap R) \geq \lambda \nu(B_D)\frac{\nu(E\cap B_D)}{\nu(B_D)}
	\]
	with $\nu(B_D) = \pi(B_D)/\pi(R)>0$. On $R'$ condition~\eqref{eq:rgeom1} holds true with $C = D$. Thus, we may assume ~\eqref{eq:rgeom1} (otherwise, replace $R$ with $R'$), which implies 
	 \begin{equation}   \label{eq:pseudo_Geo_Erg}
	    \norm{\nu P^n - \pi}_{\mathsf{TV}} \leq C \delta^n
	\end{equation}
   for all $n\in \N$. By Theorem~9 in Nummelin and Tweedie \cite{NummelinTweedie}, Eq.~\eqref{eq:pseudo_Geo_Erg} implies 
   \begin{equation}
	\int_{\R^k} \Bigl( (P-\lambda \1_R\otimes \nu)^n \1_R\Bigr)( x)  \nu(\dd x) =  \Or(r_0^{-n}).
  \end{equation}
  for some $r_0>1$. An induction over $n$ yields, for every probability measure $\mu$, 
  \begin{equation} \label{eq:distr_tau}
	     \Prob_{\mu^\ast} (\Tilde{\tau}_1 = n) = \int_{\R^k}\mu(\dd x) (P - \lambda \mathds{1}_R \otimes \nu)^n \lambda \mathds{1}_R (x).
\end{equation}
  It follows that $\tilde \tau_1$ has finite exponential moments under $\P_{\nu^\ast}$. Turning to $\mu = \pi$, we recall that the invariant probability measure $\pi$ satisfies $\pi(B) = m (B) \pi(R)$ with 
 \[
	m(B) = \sum_{n=0}^\infty \int_{\R^k} \nu(\dd x) \bigl( (P- \lambda\1_R\otimes \nu)^n \1_B\bigr)(x) 
\] 
see Nummelin~\cite[Section 2]{Nummelin78}, hence 
\begin{align*}
	\int_{\R^k} \pi(\dd x) \bigl( (P- \lambda\1_R\otimes \nu)^n \1_B\bigr)(x) &= \pi(R) \sum_{\ell=n}^\infty \int_{\R^k} \nu(\dd x) \bigl( (P -\lambda \1_R\otimes \nu)^\ell  \1_B\bigr)(x)  \\
	& = \Or (r_0^{-n}). 
\end{align*} 
We combine with Eq.~\eqref{eq:distr_tau} for $\mu=\pi$ and conclude that $\tilde \tau_1$ has finite exponential moments under $\Prob_{\pi^\ast}$. 
\end{proof} 

\subsection{The splitting of Markov random walk} \label{sec:splitMX} 

When the minorisation condition~\eqref{eq:minor} holds true for $r\geq 2$, it is natural to apply the splitting construction to the  \emph{skeleton chain} $(M_{nr})_{n\in \N_0}$. The trouble is that $(X_n)_{n\in \N_0}$ is not a deterministic function of the skeleton chain. To circumvent this difficulty, we keep track of the increments and apply the splitting construction to the chain $(M_{rn}, X_{rn}- X_{r(n-1)})$, $n\in \N_0$, with suitably chosen $r$, loosely following Niemi and Nummelin \cite{Niemi,NiemiNummelin}.

First, we need some notation. As before, let $M_n=(Z_{n-k+2},\ldots,Z_{n+1})$, $n\in \N_0$ be a Markov chain with state space $\R^k$ and transition kernel $P$, and let $X_0=0$, $X_n= Z_1+\cdots+ Z_n$. Define the kernel $K\colon \R^k\times \B_{\R^k\times \R}\to [0,1]$ by 
\[
	K(m, A\times B) := \int_{\R^k}  \1_A(z_1,\ldots, z_k) \1_B(z_{k-1}) P(m, \dd z_1\times \cdots \times \dd z_k).
\] 
The chain $(M_n,Z_n)$ is a Markov chain as well and 
\[
	\P( M_n \in A, Z_n \in B\mid M_{n-1}, Z_{n-1}) = K(M_{n-1}, A\times B)
\] 
almost surely. More generally, define $K^{\ast n}$, $n\in \N$, recursively by 
\begin{equation*}
	K^{\ast n} (m, A\times B) = \int_{\R^k\times \R} K^{\ast(n-1)} \bigl(m, \dd m' \times \dd z' \bigr) K (m', A \times (B-z')),
\end{equation*}
then for all $n\in \N$
\[
	\P\bigl(M_{rn} \in A, X_{rn} - X_{r(n-1)}\in B\mid M_{r(n-1)}, X_{r(n-1)}- X_{ r(n-2)}\bigr) = K^{*r}(M_{r(n-1)}, A \times B).
\] 
The chain $(M_n)$ is irreducible (see Section~\ref{sec:splitM}) and the kernel $K$ is spread out by Assumption~\ref{Massu:spread_out}. 
The minorisation lemma proven by Niemi \cite{Niemi} implies  that there exists an integer $r\geq 1$, a set $R$ with $\pi(R)>0$, a constant $\lambda\in (0,1)$ and an interval $\Gamma$ with positive Lebesgue measure $|\Gamma|>0$, such that for all $m \in \R^k$,
\begin{equation} \label{eq:spread_out_minor}
	K^{\ast r}(m, A\times B) \geq \lambda \nu(A) \eta(B), \quad \eta(B) \coloneqq \frac{|B\cap \Gamma|}{|\Gamma|},
\end{equation}
where $\nu(A) = \pi(A\cap R) / \pi(R)$  as in \eqref{eq:minor}.  For $n\in \N$, define
\[
	(\tilde M_n, \tilde Z_n):=  (M_{rn}, X_{rn} - X_{r(n-1)}) = (M_{rn}, Z_{r(n-1)+1}+\cdots + Z_{rn})
\] 
For $n=0$ we set $\tilde M_0 =0$ and $\tilde Z_0=0$ (which amounts to $X_{-r}:= X_0 =0$). (The choice $\tilde Z_0=0$ is arbitrary; other choices would work just as well because the transition kernel $K$ depends on $M_0$ only. ) Then $(\tilde M_n, \tilde Z_n)_{n\in \N_0}$ is a Markov chain with transition kernel $K^{\ast r}$ and, by~\eqref{eq:spread_out_minor}, it satisfies a minorisation condition with a minorising measure in product form $\nu \otimes \eta$. 

Before we apply the splitting construction to the chain $(\tilde M_n, \tilde Z_n)$, we check that the chain inherits the geometric ergodicity from $(M_n)_{n\in \N_0}$.

\begin{lem}  \label{lem:mwgeom} 
	Let $\tilde P$ be the kernel on $\R^k\times \R$ given by $\tilde P((m,z), C) = K^{\ast{r}}(m,C)$ and let $\tilde \pi$ be the measure on $\R^k\times \R$ given by $\tilde \pi(C) = \int_{\R^k} \pi(\dd m) K(m,C)$. Then there exists $\delta (0,1)$ and, for every $(m,z)\in \R^k\times \R$, a constant $C(m,z)<\infty$ so that 
	\[
		\forall n\in \N\colon\quad || {\tilde P}^n((m,z),\cdot) - \tilde \pi ||_\mathsf{TV}\leq C(m,z) \delta^n.
	\] 
\end{lem} 

\begin{proof}
	First, we check that  $\tilde \pi$ is invariant.  To see this, first we observe that $K^{\ast r}(m,A \times \R) = P^r (m,A)$ for all $m\in \R^k$ and $A\in \B_{\R^k}$, hence $\tilde \pi(A\times \R) = \int \pi(\dd m) K^{\ast r}(m,A\times \R) = \pi P^ r(A) =\pi(A)$ and 
	\[
		\tilde \pi \tilde P(C) = \int   \tilde \pi( \dd m\times \dd z) K^ {\ast r}(m,C) = \int \pi(\dd m) K^{\ast}(m,C) =\tilde \pi(C). 
	\] 
	A straightforward induction over $n$ yields ${\tilde P}^n((m,z),A\times \R) = P^{nr}(m,A)$ for all $n\in \N_0$, $m\in \R^k$, $A\in \B_{\R^k}$. For $n=1$ this follows from the definition of $\tilde P$ and our earlier observation $K^{\ast r}(m,A \times \R) = P^r (m,A)$. For the induction step from $n-1$ to $n$, notice 
	\begin{align*}
			{\tilde P}^{n}((m,z), A \times \R)  & = \int {\tilde P}^{n-1}( (m,z), \dd m'\times \dd z') K^{\ast r}(m', A \times \R) \\
	& = \int P^{(n-1)r}(m, \dd m')P^r (m',A) = P^{rn}(m,A).
\end{align*}
Consequently, we have
\begin{align*}
	|{\tilde P}^n((m,z),C) - \tilde \pi(C) | & = \Bigl| \int  {\tilde P}^{n-1} ((m,z),\dd m' \times \dd z') K^{\ast r}(m',C) - \tilde \pi(C) \Bigr| \\
		& = \Bigl| \int P^{r(n-1)} (m,\dd m') K^{\ast r} (m',C) -\int \pi(\dd m') K^{\ast r} (m',C) \Bigr| \\
	& \leq 2 \norm{ P^{r(n-1)}(m,\cdot) - \pi }_{\mathsf{TV}},
\end{align*}
where we have used $|K^{\ast r}(m',C)|\leq 1$ and the functional definition of the total variation norm.
With this we conclude that $(\tilde M_n, \tilde Z_n)$ is geometrically ergodic if $(M_n)$ is.
\end{proof} 

We may now apply the splitting technique from Section~\ref{sec:splitM} to the chain $(\tilde M_n, \tilde Z_n)$. The regeneration set is $R\times \Gamma$. The construction yields a measurable space $(\Omega,\F)$, probability measures $\P_\mu$ on $(\Omega,\F)$ (one for each $\mu \in \R^k$), a process $(\tilde M_n,\tilde Z_n, \gamma_n)$, $n\in \N_0$, with state space $\R^k\times\R\times \{0,1\}$, random times $\tau_j\colon\Omega\to \N_0\cup \{\infty\}$  so that:
\begin{itemize}
	\item Under each $\P_\mu$, the chain $(\tilde M_n, \tilde Z_n,\gamma_n)$, $n\in \N_0$, is a time-homogeneous Markov chain, the law of $(\tilde M_0,\tilde Z_0)$ is $\mu\otimes \delta_0$, moreover the marginal chain $(\tilde M_n,\tilde Z_n)$  is a Markov chain with kernel $\tilde P$ (identified with $K^{\ast r}$). 
	\item Each $\tau_j$ is a stopping time with respect to the natural filtration $(\F_n)_{n\in \N_0}$ of $((\tilde M_n, \tilde Z_n,\gamma_n))_{n\in \N_0}$. Moreover $\tau_j$ is $\P_\mu$-almost surely finite, for each $\mu$. 
	\item For all $j\geq 1$ and all $A\in \B_{\R^k}$, $B\in \B_\R$, 
	\begin{equation} \label{eq_regeneration_MX}
		\P_\mu\bigl( \tilde M_{\tau_j}\in A, \tilde Z_{\tau_j}\in B\mid \F_{\tau_j-1}\bigr) = \nu(A) \eta(B)
	\end{equation} 
	$\P_\mu$-almost surely, for every $\mu$. 
\end{itemize}
The precise definition of the stopping times is $\tau_0:=0$ and then inductively 
\begin{equation} \label{eq:def_tau}
	\tau_{j+1}:= 1+ \inf \{n\geq \tau_j\colon\  \gamma_n =1 \}.
\end{equation} 
The geometric ergodicity from Lemma~\ref{lem:mwgeom} and arguments similar to the proof of Lemma~\ref{lem:tau_geom_tail} show that the regeneration times have exponential tails: there exists $\eps>0$ so that
\begin{equation} \label{eq:mwtail} 
	\E_\nu \bigl[ \exp(\eps \tau_1)\bigr] < \infty,\quad \E_\pi \bigl[ \exp(\eps \tau_1)\bigr] < \infty. 
\end{equation} 

\section{Decay of correlation of the point process}\label{sec:ProofNewMain} 

To prove Theorem \ref{thm:NewMain}, we prove a Markov renewal theorem first for strongly aperiodic chains (Section \ref{sec:renewal}) and then without assuming strong aperiodicity (Section~\ref{sec:genproof}). In Section~\ref{sec:NewMainProof}, we apply the Markov renewal theorem to prove Theorem~\ref{thm:NewMain}. 

\subsection{A Markov renewal theorem} \label{sec:renewal} 
Let 
\[
	U_\mu(A\times B) = \E_\mu\Bigl[ \sum_{n=0}^\infty \1_{\{M_n\in A,\, X_{n } \in B\}}\Bigr]
\] 
be the renewal measure of the Markov random walk when $M_0$ has law $\mu = \pi$ (the invariant measure of the driving chain $(M_n)$) or $\mu = \nu$ (the measure in the minorisation conditions~\eqref{eq:minor} in the strongly aperiodic case or~\eqref{eq:spread_out_minor} in general). For the proof of Theorem~\ref{thm:NewMain}, we only need $A= \R^k$, but the more general statement is of interest in its own right. 

\begin{thm} \label{thm:renewal}
	There exists $\eps>0$ such that, for every $q>0$ and every Borel set $A\subset \R^k$ as $t \to \infty$, we have for $\mu \in \{ \pi, \nu \}$
	\[
			U_\mu( A\times [t-q,t]) = \pi(A) \frac{q}{\E_\pi[Z_1]} + \Or( \e^{- \eps t})
	\]
	and for $t\to - \infty$,
	\[
		U_\mu\bigl(A\times [t-q,t]\bigr) = \Or(\e^{-\eps|t|}).
	\] 
\end{thm} 

We prove the theorem first under the additional condition that the transition kernel $P(m,A)$ of the driving chain $(M_n)$ is strongly aperiodic, i.e., that in Eq.~\eqref{eq:minor} we may take $r=1$. 
(In our concrete situation, this implies $k=1$, i.e.\ $M_n=Z_{n+1}$, but we shall not exploit this explicitly.) The proof of the general case is postponed to Section~\ref{sec:genproof}. 
The fact that the interval $[t-q,t]$ is closed does not play a role, analogous statement can be made for open or half-open intervals.

For the proof we work with the split chain $(M_n^\ast)_{n\in \N_0}$ and the stopping times $\tilde \tau_j$ from Lemma~\ref{lem:splitM}. To lighten notation, we write $\Prob_\mu$ instead of $\Prob_{\mu^\ast}$, and we drop the tildas from the stopping times. In addition, define $\tau_0:=0$. The chain $(M_n^\ast)_{n\in \N_0}$ has state space $\R^k\times \{0,1\}$, let $M_n = (Z_{n-k+2},\ldots, Z_{n+1})\in \R^k$ be the marginal. By Lemma~\ref{lem:splitM}, $(M_n)_{n\in \N_0}$ is a Markov chain with transition kernel $P$, moreover the conditonal distribution of $M_{\tau_j}$ given $\tau_j-1$ is $\nu$. 

For $Z_0 \sim \nu$, the embedded chain $(X_{\tau_j })_{j\in \N}$ is a random walk with i.i.d.\ increments, the distribution is that of $Z_1+ \cdots + Z_{\tau_1}$ under $\P_\nu$.  For $M_0 \sim \pi$, the first increment $X_{\tau_1}$ has instead the distribution of $Z_1+\cdots + Z_{\tau_1}$ under $\P_\pi$. 

We first rewrite the renewal measure in terms of the embedded random walk by grouping terms between successive renewal times (this is sometimes called \emph{cyclic decomposition}, see e.g.\ \cite{Alsmeyer97}). We have 
\begin{align*} 
	U_\mu(A\times [t-q,t]) & = \E_\mu\Bigl[ \sum_{n=0}^{\tau_1-1} \1_{\{ M_n \in A,\, X_{n }\in [t-q,t]\}}\Bigr] + \sum_{j=1}^\infty \E_\mu  \Bigl[ \sum_{n=\tau_j}^{\tau_{j+1}-1} \1_{\{ M_n \in A,\, X_{n }\in [t-q,t]\}}\Bigr]. 
\end{align*} 
The first contribution on the right is kept as it is; the second contribution is rewritten in the next lemma. 
Define 
\begin{equation}\label{eq:cyclic}
	h_\mu^A(t) := \E_\mu\Bigl[ \sum_{n=0}^{\tau_1-1} \1_{\{ M_n \in A,\, X_{n }\in [t-q,t]\}}\Bigr].
\end{equation} 

\begin{lem} \label{lem:cyclic}
	For $\mu \in \{ \pi, \nu \}$: 
	\[
		U_\mu(A\times [t-q,t]) = h_\mu^A(t) + \sum_{j=1}^\infty \E_\mu \bigl[ h_\nu^A(t - X_{\tau_j })\bigr].
	\] 
\end{lem} 
\noindent Notice that, on the right side, for $\mu =\pi$ it is $h_\pi^A(t)$ for the first contribution but $h_\nu^A(t)$ inside the expected value. 

\begin{proof}
	Denote $(\F_n)_n$ the natural filtration of  the split chain $(M_n^\ast)_{n\in \N_0}$.  Remember from the proof of Lemma~\ref{lem:splitM} that each $\tau_j$ is a $(\F_n)_{n\in \N_0}$-stopping time.  By the strong Markov property for $(M_n^\ast)_{n\in \N_0}$ and regeneration property \eqref{eq:regeneration_M},
	\begin{equation*} 
		\E_{\mu} \Bigl[ \sum_{n=\tau_j}^{\tau_{j+1}-1} \1_{\{ M_n \in A,\, X_{n } - X_{\tau_j }\in [t-q- X_{\tau_j },t- X_{\tau_j }]\}} \mid \F_{\tau_j-1} \Bigr] 
		 = h_\nu^ A(t- X_{\tau_j} )
	\end{equation*} 
	$\P_\mu$-almost surely, and the lemma follows right away from the cyclic decomposition~\eqref{eq:cyclic}. 
\end{proof} 

\begin{lem} \label{lem:EeXtau}
	There exists $\eps>0$ such that for all $\delta \in (-\eps,\eps)$, 
	\[
		\E_\pi\bigl[ \exp\bigl( \delta X_{\tau_1 }\bigr)\bigr]< \infty, \quad \E_\nu \bigl[ \exp( \delta X_{\tau_1 })\bigr] < \infty.
	\] 	
\end{lem} 

\noindent That is, the increments of the random walk $(X_{\tau_j})_{j\in\N_0}$ have finite exponential moments. 

\begin{proof} 
	The lemma follows from a Cauchy-Schwarz inequality, the geometric ergodicity of the chain $(M_n)_{n\in \N_0}$, and 
Assumption~\ref{MAssu:ExpMoment}] on exponential moments.  By Assumption~\ref{MAssu:ExpMoment}, the growth rate 
	\begin{equation} \label{eq:growth}
		\alpha(\delta) \coloneqq \limsup_{n\to \infty} \frac{1}{n} \log \E_{\pi} [\e^{\delta X_n}]
	\end{equation}
        is finite for all $\delta$ in some open neighborhood $(-\epsilon,\epsilon)$ of $0$. The function $\alpha(\cdot)$ is convex and finite on $(-\epsilon,\epsilon)$, hence continuous on $(-\eps,\eps)$, in particular $\alpha(\delta)\to \alpha(0) =0$ as $\delta \downarrow 0$.  By Lemma~\ref{lem:tau_geom_tail}, the regeneration time has an exponential tail: There exist constants $C,\beta>0$ such that 
    \begin{equation}      \label{eq:Kendall}
        \Prob_\pi (\tau_1= n) \leq C\e^{-\beta n}
    \end{equation}
for all $n\in \N$. 
    By Cauchy-Schwarz, for $\delta >0$ 
    \begin{align*}
        \E_\pi[\e^{\delta X_{\tau_1 }}]  & = \sum_{n=1}^{\infty} \E_\pi[\e^{\delta X_{\tau_1 }} \mathds{1}_{\{ \tau_1 = n \}}] 
    		\leq	\sum_{n=1}^{\infty} \E_\pi[\e^{2\delta X_{n }}]^{1/2} \Prob_\pi (\tau_1 = n)^{1/2}.
    \end{align*}
	Choose $\delta$ small enough so that $\alpha(2\delta) <\beta$, then each summand goes to zero exponentially fast as $n\to \infty$, hence the sum is finite. This proves the lemma for the initial law $\pi$. The statement for initial law $\nu$ follows, since  $\nu(A)= \pi(A\cap R) / \pi (R)$ and 
	\[
		\E_\nu\bigl[\exp(\delta X_{\tau_1})\bigr] = \frac{1}{\pi(R)}\E_\pi \bigl[ \1_{R}(M_0)\, \exp(\delta X_{\tau_1})\bigr] <\infty.  \qedhere
	\] 
\end{proof}

\begin{lem}\label{lem:hdecay}
	There exists $\delta>0$ so that as $|t|\to \infty$, 
	\[
		h^A_\pi(t) = \Or( \e^{-\delta |t|}), \quad h^A_\nu(t) = \Or(  \e^{-\delta|t|}). 
	\]
\end{lem} 

\begin{proof}
Since $\nu(\cdot) = \pi(\cdot \cap R) / \pi(R)$, it suffices to consider the initial law $\pi$.
We bound
\[
	h_\pi^A(t)  = \E_\pi\Bigl[ \sum_{n=0}^{\tau_1-1} \1_{\{M_n\in A, X_{n } \in [t-q,t]\}}\Bigr]  
	\leq \sum_{n=0}^\infty \P_\pi(X_{n }\in [t-q,t], \tau_1 > n). 
\]
For $t>2q$, we further bound 
\[
	h_\pi^A(t) \leq \sum_{n=0}^\infty \P_\pi\Bigl(X_{n }\geq \frac{t}{2}\Bigr)^{1/2} \P_\pi( \tau_1 > n)^{1/2}.
\]
By Assumption \ref{MAssu:ExpMoment} and a Chernoff bound, we have for all $\epsilon>0$ and all sufficiently large $n$,
\[
	\P_\pi\Bigl(X_n\geq \frac{t}{2}\Bigr) \leq \E_\pi \bigl[\e^{\epsilon X_{n}}\bigr] \e^{-\epsilon t/2} \leq \e^{\alpha(\epsilon)n +o(n) } \e^{-\epsilon t/2}
\]
with the growth rate $\alpha$ defined in \eqref{eq:growth}.
By continuity of $\alpha$, for sufficiently small $\epsilon$ the growth of $\P_\pi(X_n\geq \frac{t}{2})$ in $n$ is dominated by the decay of $\P_\pi( \tau_1 > n)^{1/2}$, leaving the exponential decay in $t$. The case of negative $t$ follows analogously.
\end{proof}

We now give the proof of Theorem \ref{thm:renewal} for the case of strongly aperiodic driving chains; the general case is treated in the next subsection.

\begin{proof}[Proof of Theorem~\ref{thm:renewal} for strongly aperiodic driving chain]	
	The theorem follows from Stone's decomposition applied to the embedded random walk $Y_j: = X_{\tau_j }$.
    We adapt the proof of Asmussen \cite[Theorem VII.2.10]{Asmussen}.
	Under $\P_\nu$, the chain $(Y_j)$ has independent increments that are i.i.d.\ except possibly for the first one. Indeed, 
    by the strong Markov property for the split chain $(M^\ast_n)_{n\in \N_0}$ from Section~\ref{sec:splitM} and by the regeneration property~\eqref{eq:regeneration_M}, we have for $\mu \in \{\pi,\nu\}$ and all Borel sets $A_1,A_2\subset \R$
    \begin{align*}
        \P_\mu \bigl(  X_{\tau_1}\in A_1, X_{\tau_2} -  X_{\tau_1} \in A_2\bigr)
         & = \E_\mu \Bigl[ \1_{\{Z_1+\cdots +Z_{\tau_1}\in A_1\} }  \P_\mu \bigl (Z_{\tau_1+1}+ \cdots +  Z_{\tau_2}\in A_2\mid \F_{\tau_1-1} \bigr) \Bigr] \\
         & = \E_\mu \Bigl[ \1_{\{X_{\tau_1}\in A_1\} }  \P_\nu \bigl (Z_{1}+ \cdots + Z_{\tau_1}\in A_2 \bigr) \Bigr]\\
         & = \P_\mu(X_{\tau_1}\in A_1) \P_\nu(X_{\tau_1} \in A_2), 
    \end{align*}
    hence the first two increments are independent. 
    Here we have used the fact that $M_{\tau_1-1} = Z_{\tau_1}$, which motivates the index choice. 
    A similar argument yields the independence of all increments, moreover they are i.i.d.\ except for the first one when $\mu =\pi$.

    The embedded walk $(Y_j)$ is spread out  with positive drift $m:=\E[Y_1]>0$. Its increments have finite exponential moments by Lemma~\ref{lem:EeXtau}.
	Therefore, by Stone's theorem \cite{Stone1966}, the renewal measure 
	\[
		U_0(B) = \delta_0 + \E_\nu\Bigl[ \sum_{j=1}^\infty \1_{\{Y_j\in B\}}\Bigr]
	\] 
	is the sum $U_0 = U_1+U_2$ of a finite measure $U_2$ and an absolutely continuous measure $U_1$ that has a continuous density $u_1$, and in addition there exists $\theta>0$ so that, as $t\to \infty$, 
	\[ 
		U_2\bigl([t,\infty) \bigr)= \Or(\e^{-\theta t}),\quad U_2\bigl((-\infty,-t] \bigr) = \Or(\e^{-\theta t})\]
	and 
	\[
		u_1(t) = \frac{1}{m}+\Or(\e^{-\theta t}),\quad u_1(-t) = \Or(\e^{-\theta t}). 
	\]
	By Lemma~\ref{lem:cyclic}, we can write $U_\nu(A\times [t-q,t])$ as convolution,
	\[
		U_\nu(A\times [t-q,t]) = h_\nu^A(t) + \sum_{j=1}^\infty \E_\nu \Bigl[ h_\nu^A(t- Y_j)\Bigr] 
		= \int_{-\infty}^\infty h_\nu^A(t-s)U_0(\dd s). 
	\]
	For the convolution $h_\nu^A* U_0$, we plug in Stone's decomposition.  
	Write $|h_\nu^A(t)|\leq C \exp( - \delta|t|)$ for all $t\in \R$. 
	We may assume without loss of generality that $\delta<\theta$ so that, in particular, $\int_\R \exp( \delta s) U_2(\dd s) <\infty$. 	
    Consider first the case $t>0$. 
	The contribution from $U_2$ is bounded by $C$ times 
	\begin{align*}
		 & \int_{-\infty}^\infty \e^{ - \delta |t-s|} U_2(\dd s) = \e^{-\delta t} \int_{-\infty}^t \e^{\delta s} U_2(\dd s)  + \e^{\delta t} \int_t^\infty \e^{- \delta s} U_2(\dd s) \\
		 & \qquad \leq \e^{-\delta t} \int_{-\infty}^\infty \e^{\delta s} U_2(\dd s) + \e^{\delta t} U_2\bigl([t,\infty) \bigr) \\
		 & \qquad = \Or( \e^{-\delta t}) + \Or(\e^{- (\theta-\delta) t}). 
	\end{align*}
	The contribution from $U_1$ is (remember $t>0$) 
	\[
		\int_\R h_\nu^A(t-s) U_1(\dd s) = \frac 1m \int_\R h_\nu^A(t-s) \dd s + \int_\R h_\nu^A(t-s) \bigl( u_1(s) - \frac 1 m \bigr) \dd s. 
	\]
	The first term on the right is simply $\frac 1m \int_\R h_\nu^A(t') \dd t'$. The second term is split into an integral over $s< 0$ and an integral over $s>0$. The integral over $\R_+$ goes to zero exponentially fast as $t\to \infty$ by arguments similar to $U_2$. The integral over $\R_-$ is bounded by 
	\[
		 C \Bigl( \frac 1m + ||u_1||_\infty\Bigr)  \int_{-\infty}^0  \e^{- \delta(t-s)} \dd s \leq \mathrm{const}\, \e^{-\delta t}. 
	\] 
	Thus, we have shown that as $t\to \infty$, for some $\eps>0$,
	\begin{equation}\label{eq:unu1}
		U_\nu(A\times [t-q,t]) = \frac{1}{m} \int_\R h_\nu^A(t') \dd t' + \Or( \e^{-\eps t}).
	\end{equation}
	A similar reasoning yields 
	\begin{equation} \label{eq:unu2}
		U_\nu(A\times [t-q,t]) = \Or(\e^{-\eps |t|})\quad( t\to - \infty).  
	\end{equation} 
	It remains to compute $m= \E_\nu[Y_1]$ and $\int_\R h_\nu^A(s) \dd s$. The integral is 
	\begin{align*}
		\int_\R h_\nu^A(t) \dd t & = \int_\R 	\E_\nu\Bigl[ \sum_{n=0}^{\tau_1-1} \1_{\{ M_n \in A,\, X_{n }\in [t-q,t]\}}\Bigr] \dd t  = q\, \E_\nu\Bigl[ \sum_{n=0}^{\tau_1-1} \1_{\{M_n\in A\}}\Bigr]
	\end{align*} 
	since $X_n\in [t-q,t]$ if and only if $t\in [X_n, X_n+q]$. The drift is 
	\[
		m = \E_\nu\Bigl[ \sum_{n=0}^{\tau_1-1} Z_n \Bigr].
	\] 
	By the general theory of Markov chains \cite[Theorem 10.2.1]{MeynTweedie}, one knows that the invariant measure is given by
	\begin{equation} \label{eq:pinu}
		\pi (A) = \frac{\E_\nu[ \sum_{n=0}^{\tau_1-1} \1_{\{M_n\in A\}}] }{\E_\nu[\tau_1]}. 
	\end{equation} 	
	Since $Z_n$ is a function of $M_n$, we also get $\E_\pi [Z_1] = \E_\nu[ \sum_{n=0}^{\tau_1-1} Z_n] / \E_\nu[Z_n]$. It follows that 
	\begin{equation} \label{eq:h_integral}
		\frac 1m \int_\R h_\nu^A(s) \dd s = \frac{\pi(A)}{\E_\pi[Z_1]}. 
	\end{equation}
	Together with Eqs.~\eqref{eq:unu1} and~\eqref{eq:unu2}, this completes the proof of the theorem for $M_0 \sim \nu$. 
	
	For $M_0 \sim \pi$, we consider a \emph{delayed} random walk, where the distribution of $Y_1$ is $\bar{\pi} (\cdot)= \P_{\pi}(X_{\tau_1 } \in \cdot)$, while the distribution of $Y_j$ for $j>1$ remain unchanged.
	In this case, the corresponding renewal measure has the form
	\[
		U_{\pi} (B) = \E_{\pi} \biggl[ \sum_{j=1}^{\infty} \1_{\{Y_j \in B\}} \biggr] = \bar{\pi} \ast U_0 (B).
	\]
	Then, by Lemma~\ref{lem:cyclic},
	\[
		U_\pi\bigl(A\times [t-q,t]\bigr) = h_\pi^A(t) + h_\nu^A \ast \bar{\pi} \ast U_0 (t).
	\]
	The function $h_\pi^A(t)= O(\exp(-\delta|t|))$ goes to zero exponentially fast as $|t|\to \infty$, by Lemma~\ref{lem:hdecay}.
	For the convolution, we can apply the same analysis as above with $h_\nu^A$ replaced by $h_\nu^A \ast \bar{\pi}$.
	To do so, we need to check that $h_\nu^A \ast \bar{\pi}(t)$ decays exponentially in $t$ and $\int h_\nu^A \ast \bar{\pi} (t) \dd t = \int h_\nu^A(t) \dd t$.

	For the exponential decay, we first observe that by Lemma~\ref{lem:EeXtau} we have for all $t>0$ and some $\delta>0$ 
	\[
		\bar{\pi}\bigl([t,\infty)\bigr) = \P_{\pi}(X_{\tau_1}  \geq t) = \Or(\e^{-\delta t}), \quad \bar{\pi}\bigl((-\infty,-t]\bigr) = \P_{\pi}(X_{\tau_1}  \leq -t) = \Or(\e^{-\delta t}).
	\] 
	Furthermore, $\bar{\pi}$ is finite.
	Hence, we can apply the same analysis as we did to the contribution of $U_2$ above and conclude that $h_\nu^A \ast \bar{\pi}(t)$ decays exponentially in $t$.

	By Fubini's theorem,
	\begin{equation*}
		\int_\R h_\nu^A \ast \bar{\pi} (t) \dd t = \int_\R h_\nu^A(t-s)\bar{\pi}(\dd s)\dd t = \int_\R \int_\R h_\nu^A(t) \dd t \bar{\pi}(\dd s) = \int_\R h_\nu^A(t) \dd t.
	\end{equation*}
	This concludes the proof for  $M_0 \sim \pi$.
\end{proof} 

\subsection{Proof of the renewal theorem for driving chains that are not strongly aperiodic} 
\label{sec:genproof}  

Next, let us turn to the general case where in the minorisation condition we may have $r\geq 2$. 
We reduce the proof to the previous case by finding appropriate functions $H^A_\mu$ such that an analogue of Lemma~\ref{lem:cyclic} holds with $H_\mu^ A$ replacing $h^A_\mu$.
Then, we check that $H^A_\mu$ satisfies the same properties of $h^A_\mu$ that are required in the proof of Theorem \ref{thm:renewal} for the strongly aperiodic driving chain case.

We first express the renewal measure $U_\mu$ in terms of the skeleton chain $(M_{rn},X_{rn } - X_{r(n-1) })_{r\in \N_0}$. We start from 
\begin{align*}
	U_\mu\bigl(A\times  [t-q,t]\bigr) & = \E_\mu\Bigl[ \sum_{n=0}^{r-1} \1_{\{ M_n\in A, X_{n }\in [t-q,t]\}}\Bigr]+  \sum_{j=1}^\infty  \E_\mu\Bigl[ \sum_{n=jr}^{(j+1)r-1} \1_{\{M_n\in A, X_{n }\in [t-q,t]\}}\Bigr].
\end{align*}
For $m \in \R^k$, let 
\[
	g^A(m,t):= \E_{\delta_{m}}\Bigl[ \sum_{n=0}^{r-1} \1_{\{M_n\in A, X_{n }\in [t-q,t]\}}\Bigr]. 
\] 
Then 
\begin{align*}
	&\E_\mu\Biggl[ \sum_{n=jr}^{(j+1)r-1} \1_{\{M_n\in A, X_{n }\in [t-q,t]\}}\Biggr]\\
	&\qquad  = \E_\mu\Biggl[ \E_\mu\Bigl[ \sum_{n=0}^{r-1} \1_{\{M_{jr+n}\in A, X_{jr +n}-X_{jr }\in [t-X_{jr }-q, t- X_{jr }]\}} \, \Big|\, M_\ell, \ell \leq jr\Bigr]\Biggr] \\
	&\qquad = \E_\mu \Bigl[ g^A(M_{jr},t - X_{jr })\Bigr]. 
\end{align*} 
Thus, 
\begin{equation} \label{eq:umurewrite}
	U_\mu\bigl(A\times [t-q,t]\bigr) = \E_\mu\Bigl[ \sum_{j=0}^\infty g^A(M_{jr}, t- X_{jr })\Bigr].
\end{equation}
Recall that $(M_{nr},X_{rn } - X_{r(n-1) })_{n\in\N_0}$ is a strongly aperiodic Markov chain with transition kernel $K^{\ast r}$. We may now apply the splitting construction for the Markov random walk from Section~\ref{sec:splitMX}. Thus let $\tilde M_n$, $\tilde Z_n$ and $\tau_j$ be as in Section~\ref{sec:splitMX} and 
\[
    \tilde X_0:=0,\quad \tilde X_n:= \tilde Z_1+\cdots + \tilde Z_n\quad (n\in \N).
\]
Recall that $(\tilde M_n, \tilde X_n)$ is equal in distribution to $(M_{nr}, X_{nr})$.  Let
  \begin{equation*}
    	H_\mu^A(t) \coloneqq \E_\mu \biggl[ \sum_{j=0}^{\tau_1-1} g^A (\tilde M_{j}, t-\tilde X_{j }) \biggr].
   \end{equation*}
Remember the measure $\nu$ from the minorisation condition~\eqref{eq:spread_out_minor}. 
The following lemma is proven with a cyclic decomposition in the same way as Lemma~\ref{lem:cyclic}, the proof is therefore omitted. 

\begin{lem} \label{lem:H_cyclic}
    For $\mu \in \{\nu,\pi\}$ we have
    \begin{equation*}
    	U_\mu \bigl(A \times [t-q,t]\bigr) = H_\mu^A(t) + \E_\mu \biggl[ \sum_{\ell=1}^\infty H_\nu^A (t- \tilde X_{\tau_\ell }) \biggr].
    \end{equation*}
\end{lem}

We also need the following two lemmas.

\begin{lem}\label{lem:Hdecay}
	There exists $\delta>0$ so that as $|t|\to \infty$, 
	\[
		H^A_\pi(t) = \Or( \e^{-\delta |t|}), \quad H^A_\nu(t) = \Or(  \e^{-\delta|t|}). 
	\]
\end{lem} 

\begin{proof}
	Define $g_\ell^ A(m,t) = \P_{\delta_m}(M_\ell \in A, X_\ell \in [t-q,t]  )$ so that  $g^A(m,t) = \sum_{\ell =0}^{r-1} g_\ell^A(m,t)$.
	We apply  a Cauchy-Schwarz inequality and use $(g_\ell^ A)^2 \leq g_\ell^ A$, this gives 
	\begin{align*}
		H_\mu^A(t) & = \sum_{j=0}^\infty \sum_{\ell=0}^{r-1}   \E_\mu \Bigl[  \1_{\{\tau_1 > j\}} g_\ell^ A(\tilde M_j, t- \tilde X_j)    \Bigr]  \\
			 & \leq \sum_{j=0}^\infty \sum_{\ell=0}^{r-1} (\E_\mu \bigl[ g_\ell^A(\tilde M_j, t- \tilde X_j) \bigr])^{1/2} \P_\mu(\tau_1 >j )^{1/2}. 
	\end{align*} 
	We remember the chain $(\tilde M_j, \tilde X_j)$ is equal in distribution to $(M_{jr}, X_{jr})$, combine with the definition of $g_\ell^A$ and a Markov property and deduce 
	\[
		\E_\mu \bigl[ g_\ell^A(\tilde M_j, t- \tilde X_j) \bigr] = \E_\mu\bigl[  \1_{\{M_{jr+\ell} \in A, t - X_{jr+\ell}\in [t-q,t] \} }\bigr]  \leq \P_\mu (X_{jr+\ell}\in [t-q,t])
	\] 
	which gives 
	\[
		H_\mu^A(t) \leq \sum_{j=0}^\infty \sum_{\ell=0}^{r-1}  \P_\mu ( X_{jr+\ell} \in [t-q,t])^{\frac{1}{2}} \P_\mu(\tau_1 >j)^{\frac{1}{2}}. 
	\]
	The regeneration time $\tau_1$ has geometric tails by Lemma~\ref{lem:tau_geom_tail} and the desired exponential decay follows from arguments similar to the proof of Lemma~\ref{lem:hdecay}. 
\end{proof}

\begin{lem}\label{lem:H_integral}
	We have 
	\begin{equation*}
		\frac{1}{\E_\nu [\tilde X_{\tau_1 }]} \int_\R H_{\nu}^A(t) \dd t = q\, \frac{\pi(A)}{\E_\pi[Z_1]}.
	\end{equation*}
\end{lem}

\begin{proof}
	Let $g_\ell^ A(m,t)$ be as in the proof of Lemma~\ref{lem:Hdecay}. Notice 
	\[
		\int_\R g_\ell^A(m,t) \dd t = \E_{\delta_m}\Bigl[ \1_{\{M_\ell \in A, X_\ell \in [t-q,t]\}}\Bigr] \dd t = \P_{\delta_m}( M_\ell \in A). 
	\] 
	Define $f_\ell^ A(m):= \P_{\delta_m}(M_\ell \in A)$ and $f^A(m) := \sum_{\ell =0}^{r-1} f_\ell^ A(m)$.  Then 
	\[
		 \int_\R H_\nu^A(t) \dd t  = \int_\R \E_\nu \Bigl[ \sum_{j=0}^{\tau_1 - 1} \sum_{\ell=0}^{r-1} g_\ell^ A(\tilde M_j, t - \tilde X_j)     \Bigr]\dd t 
	 = q\,  \E_\nu \biggl[ \sum_{j=0}^{\tau_1-1}  f^A(\tilde M_{j}) \biggr].
	\]
	By the analogue of~\eqref{eq:pinu} for $(\tilde M_j)$, 
	\[
		\E_\nu \biggl[ \sum_{j=0}^{\tau_1-1}  f^A(\tilde M_{j}) \biggr] = \E_\pi \bigl[ f^A(\tilde M_0)\bigr] \E_\nu[\tau_1]. 
	\]
	We may replace $\tilde M_0$ with $M_0$, plug in the definition of $f^A$, combine with the Markov equality, and obtain 
	\[
		\E_\pi \bigl [ f^ A(\tilde M_0) \bigr] = \E_\pi \Bigl[\sum_{\ell =0}^{r-1} \1_{\{M_\ell \in A\}}\Bigr]  = r \P_\pi( M_\ell \in A),
	\]
	hence altogether 
	\[
		\int_\R H_\nu^A(t) \dd t = q r \pi(A) \E_\nu[\tau_1]. 
	\] 
	Turning to the drift, remember the increments $\tilde Z_n = \tilde X_n - \tilde X_{n-1}$ from Section~\ref{sec:splitMX}, equal in distribution to $Z_{rn+1}+\cdots + Z_ {(r+1)n}$. Notice that $\E_\nu[\tilde Z_0] = \E_\nu [\tilde Z_{\tau_1}]$ by the strong Markov property and because conditional on $\tau_1$, the conditional law of $M_{\tau_j}$  is $\nu$. Therefore 
	\[
		\E_\nu [\tilde X_{\tau_1}\bigr]  = \E_\nu \biggl[\sum_{n=1}^{\tau_1} \tilde Z_n \biggr] = \E_\nu \biggl[\sum_{n=0}^{\tau_1-1} \tilde Z_n \biggr]
			 = \E_\pi [\tilde Z_0] \E_\nu[\tau_1] = r\E_\pi[Z_1] \E_\nu[\tau_1].
	\]
	The lemma follows. 
\end{proof}

\proof[Proof of Theorem \ref{thm:renewal} for the general case] 
    In the proof of Theorem~\ref{thm:renewal} for strongly aperiodic chains, what matters is: (i) the embedded chain $Y_j = X_{\tau_j}$ is a spread out random walk with independent increments (i.i.d.\ except possibly for the first increment), (ii) the increments have finite exponential moments (Lemma~\ref{lem:EeXtau}), (iii) the renewal measure $U_\mu$ can be expressed in terms of the embedded random walk (Lemma~\ref{lem:cyclic}) with a function that has exponential tails (Lemma~\ref{lem:hdecay}), (iv) the integral of the auxiliary function is related to the stationary measure of the chain as in~\eqref{eq:h_integral}.
    
    These ingredients extend to the general case. The embedded random walk is $Y_j: = \tilde X_{\tau_j}$, $j\in \N_0$. By the strong Markov property for the split chain $(\tilde M_n, \tilde Z_n, \gamma_n)$ from Section~\ref{sec:splitMX} and by the regeneration property~\eqref{eq_regeneration_MX}, we have for $\mu \in \{\pi,\nu\}$ and all Borel sets $A_1,A_2\subset \R$
    \begin{align*}
        & \P_\mu \bigl( \tilde X_{\tau_1}\in A_1, \tilde X_{\tau_2} - \tilde X_{\tau_1} \in A_2\bigr) \\ \nonumber
         &\quad = \E_\mu \Bigl[ \E_\mu \bigl [ \1_{A_1}(\tilde X_{\tau_1}) \, \mathds{1}_{A_2}(\tilde Z_{\tau_1+1}+ \cdots + \tilde Z_{\tau_2}) \mid \F_{\tau_1-1} \bigr] \Bigr] \\ \nonumber
         &\quad  = \E_\mu \Bigl[ \E_\mu \bigl [ \1_{A_1}(\tilde X_{\tau_1} - \tilde X_{\tau_1-1} + \tilde X_{\tau_1-1}) \, \P_{\delta_{\tilde M_{\tau_1}}}( \tilde X_{\tau_1} \in A_2) \mid \F_{\tau_1-1} \bigr] \Bigr] \\ \nonumber
         &\quad = \P_\mu(\tilde X_{\tau_1}\in A_1)\, \P_\nu(\tilde X_{\tau_1} \in A_2). 
    \end{align*}
    Hence, the first two increments are independent. A similar argument yields the independence of all increments, moreover they are i.i.d.\ except for the first one when $\mu =\pi$. The embedded random walk is spread out because, by Assumption~\ref{Massu:spread_out}, there exists $n_0$ so that the law of $(M_{n_0},X_{n_0 })$ under $\Prob_z$ has absolutely continuous components with respect to $\pi \otimes l$, no matter the starting point $Z_0=z$. This then also applies to $(M_{rn_0}, X_{rn_0})$. By Eq.~\eqref{eq_regeneration_MX}, the embedded walk $Y_j$ has absolutely continuous increments $\tilde Z_{\tau_j} - \tilde Z_{\tau_{j-1}}$ and in particular, it is spread out.
    
    The proof of Lemma~\ref{lem:EeXtau} extends to the exponential moments of $\tilde X_{\tau_1}$ in an obvious way, whereas the analogues of Lemmas \ref{lem:cyclic}, \ref{lem:hdecay} and Eq.~\eqref{eq:h_integral} are Lemmas \ref{lem:H_cyclic}, \ref{lem:Hdecay} and Lemma~\ref{lem:H_integral}.
    The analysis in the proof of Theorem \ref{thm:renewal} for strongly aperiodic driving chains can thus be carried out for the general case.
\qed

\subsection{Proof of Theorem~\ref{thm:NewMain}} \label{sec:NewMainProof}

We use the following theorem, which connects the second factorial moment measure $M_{\Phi^{(2)}}$ of a point process with the Palm measure.

\begin{thm}[{\cite[Theorem 6.3.24]{baccelli}}]
    \label{thm:BBK}
    Let $\Phi$ be a stationary point process on $\R$ with intensity $\kappa>0$ and Palm version $\Phi^0$ and let
    \begin{equation}
        \eta(B) \coloneqq \E [\Phi^!(B)] \quad \text{with} \quad \Phi^!\coloneqq\Phi^0 - \delta_0, B \in \B(\R)
    \end{equation}
    be the \emph{reduced second moment measure} of $\Phi$.
    Then the second factorial moment measure is 
    \begin{equation}
        \E\Bigl[ \Phi^{(2)}(A\times B) \Bigr] = \kappa \int_A \eta(B-x) \dd x, \quad A,B\in\B(\R).
    \end{equation}
\end{thm}

The immediate consequence for our stationary point process is that, when $A$ and $B+t$ are two disjoint sets, 
\begin{equation}\label{eq:Apply_Renewal_measure}
    \E\bigl[ \Phi(A) \Phi(B+t)\bigr] = \frac{1}{\E[Z_1]} \int_A \E_\pi\Biggl[ \sum_{n\in \Z\setminus \{0\}} \1_{\{X_{n } \in B +t - x\}} \Biggr] \dd x. 
\end{equation}
We are interested in the case where $A$ and $B$ are bounded intervals. By translation invariance and symmetry, it is enough to consider $\inf A= 0$.  Let us also assume that the intervals are closed and write $A = [0,a_2]$, $B= [b_1,b_2]$.

The sum over $n\in \Z\setminus \{0\}$ splits into two contributions. For $t$ sufficiently large such that $0 \notin B+t-a_2$, the first is 
\begin{equation} 
    	\E_\pi\Biggl[ \sum_{n=1}^\infty \1_{\{X_{n } \in B + t-x\}} \Biggr] = U_\pi\bigl(\R^k\times (B +t -x)\bigr).
\end{equation} 
By Theorem~\ref{thm:renewal}, as $t\to \infty$, for $x\in [0,a_2]$, this becomes 
\[
	 U_\pi\bigl(\R^k\times [b_1 + t- x, b_2 +t -x]\bigr)  = \frac{b_2-b_1}{\E[Z_1]}+ \Or(\e^{-\eps (t-x)}). 
\] 
Turning to the contribution from $n\leq - 1$, if the spacings $Z_n$ are positive, then $X_n <0$ almost surely, hence, for sufficiently large $t$, the indicator that $X_n$ is in $B+ t - x \subset \R_+$ simply vanishes. If the spacings may take negative values, we note that the reversibility yields 
\[
	\E_\pi \Bigl[ \sum_{n=1}^\infty \1_{\{X_{-n } \in B+t - x\}}\Bigr] 
	 = \E_{\pi} \Bigl[ \sum_{n=1}^\infty \1_{\{X_{n }\in - B - t + x \}} \Bigr]  
	 \leq U_\pi\bigl( \R^k\times (-\infty, - t + a_2]\bigr) = \Or(\e^{-\eps t}). 
\] 
Combining all of the above, we obtain 
\[
	\E\bigl[ \Phi(A)\Phi(B)\bigr] = \frac{|A|}{\E[Z_1]}\, \frac{|B|}{\E[Z_1]} + \Or(\e^{-\eps t}).
\] 
and Theorem~\ref{thm:NewMain} is proven.

\section{Application to interacting chain of particles} \label{sec:ProofStatMech}

Here we prove Theorem \ref{thm:gibbs-markov} on Gibbs point processes. In Section~\ref{sec:transfer} we introduce the transfer operator and a Markov chain of blocks of consecutive spacings, and we check that it is geometrically ergodic. The link between a Markov chain and a Gibbs point process is established in Section~\ref{sec:canonical}. Exponential decay of correlations is deduced in Section~\ref{sec:gibbs-corr}.

\subsection{Transfer operator and Markov chain}\label{sec:transfer}
One-dimensional systems of continuum interacting particles are best treated in the constant pressure ensemble, see Ruelle \cite[Chapter 5.6]{Ruelle-book}. The associated partition function, given $\beta,p>0$, is 
\begin{equation} \label{eq:qnp}
    Q_N(\beta,p) :=  \frac{1}{(N-1)!} \int_0^\infty \Bigl\{ \int_{[0,L]^{N-1}} \exp\Bigl( - \beta p L - \beta \sum_{0 \leq i < j \leq N} v(|x_i - x_j|)\Bigr) \dd \vect x\Bigr\} \dd L
\end{equation}
where, in the integral, $x_0 = 0$ and $x_{N} = L$. It is the normalisation constant for a probability measure for a fixed number of particles, two of them pinned at  $x=0$ and $x=L$. The system's length $L$ is random, its distribution is controlled by the so-called \emph{pressure} $p>0$. 

In view of the energy's symmetry we may drop the factorial, integrate over $0< x_1<\cdots < x_{N-1} < L$, and change variables from $x_i$ to $z_i = x_i - x_{i-1}$, $i=1,\ldots, N$ so that $L = z_1+\cdots+ z_N$. If $v$ has a hard core $r_{\mathrm{hc}}>0$, finite range $R \in (r_\mathrm{hc},\infty)$, and $m\in \N$ satisfies $R < (m+1)r_\mathrm{hc}$, the integral becomes  
\[
    Q_N(\beta,p) = \int_{(r_\mathrm{hc},\infty)^N} \exp \Biggl( - \beta p \sum_{i=1}^N z_i - \beta \sum_{\substack{1\leq i <j \leq N+1\colon\\ |j-i|\leq m}} v(z_i+\cdots + z_{j-1})\Biggr) \dd \vect z.
\]
Let $k:=m-1$. Define 
\begin{align*}
    V_p(z_1,\cdots,z_k) &= \sum_{1\leq i<j\leq k+1 } v(z_i + \cdots +z_{j-1}) + p \sum_{i=1}^k z_i \\
    W(z_1,\cdots,z_k;z'_1,\ldots, z'_{k})& = \sum_{i= 1}^k \sum_{j=1}^k v(z_i+\cdots+z_k+z'_1+\cdots+z'_j)
\end{align*}
and 
\begin{equation} \label{eq:K_general}
    K_{\beta, p}(\vect z;\vect z') =  (\vect z) \exp \Bigl( - \beta \Bigl( \frac 12  V_p(\vect z) +W(\vect z;\vect z') + \frac 12 V_p(\vect z') \Bigr)\Bigr). 
\end{equation}
Then, the partition function along integer multiples of $k$
becomes 
\[
    Q_{kN}(\beta,p) 
    = \int_{((r_\mathrm{hc},\infty)^k)^{N}} \e^{- \beta  V_p(\vect z^{(1)}) /2 } \Biggl(\prod_{i=1}^{N-1} K_{\beta,p}(\vect z^{(i)}, \vect z^{(i+1)})\Biggr) \e^{- \beta  V_p(\vect z^{(N)}) /2 } \dd \vect z^{(1)}\cdots \dd \vect z^{(N)}
\]
the integral is over $((r_\mathrm{hc},\infty)^k)^{N}$. This rewrite motivates a closer look at the kernel $K_{\beta,p}$ and the integral operator
\[
    \mathcal K_{\beta,p} f(\vect z): = \int_{(r_\mathrm{hc},\infty)^k} K_{\beta,p}(\vect z,\vect z') f(\vect z') \dd \vect z'. 
\]
Let $\mathsf{s}\colon\R^k\to \R^k$ be the map that reverses the order of the spacings, $\mathsf{s}(z_1,\cdots,z_k)=(z_k,\cdots,z_1)$. Write $L^2((r_\mathrm{hc},\infty)^2)$ for the $L^2$ space with respect to Lebesgue measure.

\begin{lem} \label{lem:K_properties}
    The transfer operator with the kernel \eqref{eq:K_general} enjoys the following properties:
    \begin{enumerate}[label = (\roman*), ref = (\roman*)]
        \item \label{K_inversion} inversion symmetry: $K_{\beta,p}(x,y) = K_{\beta,p}(\mathsf{s}(y),\mathsf{s}(x))$,
        \item \label{K_strict_positive} strictly positive kernel: $K_{\beta,p}(x,y)>0$ for all $(x,y) \in (r_{\text{hc}},\infty)^{2k}$,
        \item \label{K_compact} $\mathcal{K}_{\beta,p}$ is a compact operator in $L^2((r_\mathrm{hc},\infty)^k)$. 
    \end{enumerate}
\end{lem}

\begin{proof}
The inversion symmetry follows from a straightforward calculation.
The positivity of the kernel is guaranteed by the fact that the potential is finite outside the hard core. As $v$ is bounded from below, the kernel $K_{\beta,p}$ is bounded from above by 
\[
    K_{\beta,p}(\vect z, \vect z') \leq C_\beta\exp\Bigl( - \frac 12 \beta p \Bigl(\sum_{i=1}^k z_i + \sum_{i=1}^k z'_i\Bigr)\Bigr) 
\]
for some constant $C_\beta>0$. It follows that 
\[
    \int_{((r_\mathrm{hc},\infty)^k)^2} |K_{\beta,p}(\vect z,\vect z')|^2 \dd \vect z\, \dd \vect z' < \infty
\]
hence $\mathcal K_{\beta,p}$ is Hilbert-Schmidt and (iii) follows. 
\end{proof} 

By the Krein-Rutman theorem \cite[Theorem 19.3]{Deimling}, the dominant eigenvalue $\lambda_{0}(\beta,p)$ of the transfer operator $\mathcal K_{\beta,p}$ is strictly positive and simple,  and the corresponding eigenfunction $\varphi_{0;\beta,p}$ can be chosen to be strictly positive  on $(r_\mathrm{hc},\infty)^k$ and normalised. Using these, we define a Markov transition kernel on $(r_\mathrm{hc},\infty)^k$ by 
\begin{equation}
    \label{eq:TransitionKernel}
    P_{\beta,p}(\vect z,A) = \frac{1}{\lambda_0(\beta,p) \varphi_{0;\beta,p}(\vect z)} \int_A K_{\beta,p}(\vect z,\vect z') \varphi_{0;\beta,p}(\vect z') \dd \vect z'.
\end{equation}
This Markov kernel has invariant measure $\pi$ with density proportional to $\varphi_{0;\beta,p}(\mathsf{s}(\vect z))\varphi_{0;\beta,p}(\vect z)$.

Let $(\tilde M_n)_{n\in \N}$ be a Markov chain with state space $(r_{\mathrm{hc}}, \infty)^k$ and transition kernel $P_{\beta,p}$. Write $\tilde M_n = (Z_{kn-k+1},\ldots, Z_{kn})$ and define $X_0:=0$, $X_n:= Z_1+\cdots + Z_n$. We check that the analogues of Assumptions~\ref{MAssu:RegularityMC}--\ref{MAssu:Inversion} hold true. 

The spacings are positive, thus Assumption~\ref{MAssu:Inversion} holds true. Clearly for every $n\geq 2$ and every initial value of $\tilde M_0$, the law of $(\tilde M_n, X_{nk})$ is absolutely continuous with respect to Lebesgue measure on $(r_\mathrm{hc},\infty)^k\times \R$, hence also with respect to $\pi \times l$ and the Markov random walk is spread out (see Assumption~\ref{Massu:spread_out}). Geometric ergodicity (Assumption~\ref{MAssu:RegularityMC}) is deduced from the Krein-Rutman theorem. 

\begin{lem}
    \label{lem:LJgeoErg}
	The Markov chain with transition kernel defined in \eqref{eq:TransitionKernel} is geometrically ergodic.
\end{lem}

\begin{proof}
    To lighten notation we drop the $(\beta,p)$-dependence from the notation. 
     Let $\varphi_0$ be the principal eigenfunction of $\mathcal K$ and $\psi_0:= \varphi_0 \circ \mathsf s$. Multiplicative constants are chosen so that $\varphi_0$ is strictly positive and $\int_{(r_\mathrm{hc},\infty)^k} \varphi_0 \psi_0 \dd \vect z = 1$. The probability measure $\pi$ with probability density function $\psi_0\psi_0$ is invariant for the kernel $P$. Let $\Pi$ be the integral operator given by
     \[
        \Pi f(\vect z) = \Bigl( \int_{(r_\mathrm{hc},\infty)^k} \psi_0(\vect z') f(\vect z') \dd \vect z'\Bigr) \varphi_0(\vect z). 
     \]
     Put differently, $\Pi f = \langle \psi_0,f\rangle \varphi_0$ with $\langle\cdot,\cdot\rangle$  the scalar product on $L^2((r_\mathrm{hc},\infty)^k,\mathrm{Leb})$ (square-integrable functions with respect to Lebesgue measure). 
     The operator $\Pi$ is a rank one projection that satisfies $\mathcal K \Pi = \Pi \mathcal K = \lambda_0 \mathcal K$ and thus 
    \begin{equation}
        \frac{1}{\lambda_0^n} \mathcal{K}^n - \Pi = \left( \frac{1}{\lambda_0}\mathcal{K} - \Pi \right)^n.
    \end{equation}
    for all $n\in \N$.
    We have, by definition of $P$, for all $\vect z\in (r_\mathrm{hc},\infty)^k$ and all Borel sets $A \subset  (r_\mathrm{hc},\infty)^k$,
    \begin{equation*}
        P^n(\vect z,A) - \pi(A) = \frac{1}{\varphi_0(\vect z)} \left[\left( \frac{1}{\lambda_0^n}\mathcal{K}^n - \Pi \right)\varphi_0 \mathds{1}_A \right](\vect z) 
    \end{equation*}
    It is known from the Krein-Rutman theorem that the operator $\mathcal T:= \frac{1}{\lambda_0} \mathcal K - \Pi$ in $L^2((r_\mathrm{hc},\infty)^k,\mathrm{Leb})$ has spectral radius $r$ strictly smaller than $1$
    \begin{equation}      \label{eq:SpectralRadius}
        r= \limsup_{n \rightarrow\infty} \norm{\mathcal{T}^n}^{1/n} <1.
    \end{equation}
    To go from here to geometric ergodicity, we need to pass from the $L^2$-norm to the supremum norms. 
    A straightforward calculation gives
    \begin{equation}
        C  \coloneqq \sup_{\vect z\in(r_\mathrm{hc},\infty)^k} \int K(\vect z,\vect z')^2 \dd \vect z' <\infty.
    \end{equation}
    By the Cauchy-Schwarz inequality applied to $\varphi_0(\vect z) = \lambda_0^{-1} \int K(\vect z,\vect z')  \varphi_{0}(\vect z') \dd \vect z'$, the principal eigenfunction $\varphi_0$ is bounded.   
    The operator $\mathcal T$ has integral kernel $T(\vect z,\vect z') = \frac{1}{\lambda_0} K(\vect z,\vect z') - \psi_0(\vect z') \phi_0(\vect z)$. 
    With the inequality  the inequality $(a-b)^2 \leq 2(a^2+b^2)$ we get 
    \begin{equation}
        C^\prime \coloneqq \sup_{\vect z \in (r_\mathrm{hc},\infty)^k}  \int T(\vect z,\vect z')^2 \dd \vect z' \leq 2C + 2 \norm{\varphi_0}_\infty^2 \norm{\varphi_0}_2^2 < \infty.
    \end{equation}
    It follows that 
    \begin{align*}
        |P^n(\vect z,A) - \pi (A)| & \leq \frac{1}{\varphi_0(\vect z)} || \mathcal T^n (\varphi_0 \1_A) ||_\infty \\
        &\leq \frac{1}{\varphi_0(\vect z)}\, \sqrt{C'} || \mathcal T^{n-1} (\varphi_0 \1_A) ||_{2}\\
        &\leq \frac{1}{\varphi_0(\vect z)}\, \sqrt{C'} ||\mathcal T^{n-1}||\, ||\varphi_0 \1_A||_{2}\\
        &\leq \frac{1}{\varphi_0(\vect z)}\, \sqrt{C'} ||\varphi_0||_{2}\,  ||\mathcal T^{n-1}||.
    \end{align*}    
    By~\eqref{eq:SpectralRadius}, there exists $\delta\in (0,1)$ and a constant $c>0$ so that $||\mathcal T^{n-1}|| \leq c (1-\delta)^n$. 
    It follows that $\sup_A |P^n(\vect z,A) -\pi(A)|$ is bounded by a $\vect z$-dependent constant times $(1-\delta)^n$, i.e., the chain is geometrically ergodic. 
\end{proof}

Finally, we check the condition on exponential moments from Assumption~\ref{MAssu:ExpMoment}. 
As the spacings are positive, it is enough to consider $\theta> 0$.

\begin{lem}
    \label{lem:ExpMomIncrLJ}
    For all $\theta \in (0,\beta^{-1}p/2)$, 
    \[
        \limsup_{n\to \infty} \frac 1 n \log \E_\pi\bigl[ \e^{\theta(Z_1+\cdots + Z_n)}] <\infty. 
    \]
\end{lem}

\begin{proof}
    Let us assume for simplicity that $\beta =1$. 
    We bring back the $p$-dependence in the notation for the transfer operator, the principal eigenvalue, and the left and right eigenvectors $\psi_0$ and $\varphi_0$. 
    We choose multiplicative constants in such a way that $\langle \varphi_0,\psi_0\rangle =1$ so that the stationary measure $\pi$ has probability density function $\varphi_0\psi_0$.  
    First, we rewrite the expected value we are after with the transfer operator at the perturbed pressure along integer multiples of $k$: 
    \begin{multline*}
        \E_\pi[\e^{\theta (Z_1+\cdots + Z_{nk})}] \\
        = \int \exp\Bigl( \theta \sum_{i=1}^n \sum_{j=1}^k z_j^{(i)}\Bigr) \Biggl( \prod_{i=1}^{n-1} \frac{1}{\lambda_0(p)\varphi_{0,p}(\vect z^{(i)})} K_p(\vect z^{(i)},\vect z^{(i+1)}) \varphi_{0,p}(\vect z^{(i+1)})\Biggr)   \\
        \times \psi_{0,p}(\vect z^{(1)}) \varphi_{0,p}(\vect z^{(1)}) \dd \vect z^{(1)}\cdots \dd \vect z^{(n)}.
    \end{multline*}
    The integral is over $((r_{\mathrm{hc}},\infty)^k)^{n}$. The integrand is equal to 
    \[
        \frac{1}{\lambda_0(p)^{n-1}}\, \psi_{0,p}(\vect z^ {(1)}) \exp\bigl( \theta \sum_{i=1}^k z_i^{(1)}\bigr) \Biggl( \prod_{i=1}^{n-1} K_{p-\theta}(\vect z^{(i)},\vect z^{(i+1)})\Biggr) 
        \exp\bigl( \theta \sum_{i=1}^k z_i^{(n)}\bigr) \varphi_{0,p}(\vect z^{(n)}). 
    \]
    and bounded by 
    \[
        \frac{1}{\lambda_0(p)^{n-1}}\, \psi_{0,p}(\vect z^ {(1)})  \Biggl( \prod_{i=1}^{n-1} K_{p-2\theta}(\vect z^{(i)},\vect z^{(i+1)})\Biggr) 
         \varphi_{0,p}(\vect z^{(n)}). 
    \]
    Thus, 
    \[
        \E_\pi[\e^{\theta (Z_1+\cdots + Z_{nk})}]
        \leq \frac{1}{\lambda_0(p)^{n-1}} \langle \psi_{0,p}, \mathcal K_{p-2\theta}^{n-1}\varphi_{0,p}\rangle.
    \]
   Invoking the Krein-Rutman theorem as in the proof of Lemma~\ref{lem:LJgeoErg}, we see that as $n\to \infty$, the scalar product on the right side behaves like 
   \[
       (1+o(1)) \lambda_0(p-2\theta)^{n-1} \langle \varphi_{0,p},\varphi_{0,p-2\theta}\rangle \langle\varphi_{0,p-2\theta},\varphi_{0,p}\rangle
   \]
    It follows that 
    \[
        \limsup_{n\to \infty} \frac{1}{nk} \log \E_\pi\bigl[ \e^{\theta(Z_1+\cdots+ Z_{nk)}}\bigr] \leq \frac 1 k \log \frac{\lambda_0(p-2\theta)}{\lambda_0(p)} < \infty. 
    \]
    For integers $n$ that are not necessarily multiples of $k$, we may bound  $\theta(Z_1+\cdots + Z_{n})$ by $\theta(Z_1+\cdots + Z_{k \ceil{n/k}})$. 
    The lemma follows. 
\end{proof}

\subsection{From Markov chain to Gibbs point process. Proof of Theorem~\ref{thm:gibbs-markov}(a)} \label{sec:canonical}
Next we show that, for every $\beta,\zeta >0$, there is a pressure $p= p(\beta,\zeta)>0$ so that sequence of spacings associated with the Gibbs point process $\mathsf P_{\beta,\zeta}$ comes from a Markov chain with kernel $P_{\beta,p}$; this is the first part of Theorem~\ref{thm:gibbs-markov}. The proof comes in several steps: 
\begin{enumerate}
    \item Given $p>0$ and starting from the Markov chain, show that the law of the sequence $(Z_n)_{n\in \Z}$ for the Markov chain is a Gibbs measure on $(r_\mathrm{hc},\infty)$, i.e., it satisfies the DLR equations for a suitable Hamiltonian (Lemma~\ref{lem:spacinggibbs}). 
    \item Show that the DLR equations for the stationary sequence $(Z_n)_{n\in \Z}$ imply that the the law of associated point process $\Phi$  is a \emph{canonical Gibbs measure} \cite{georgii1976canonical,georgii-canonical-book} 
    (Lemma~\ref{lem:canonical2}). 
    \item Deduce from Georgii's results \cite{georgii1976canonical,georgii-canonical-book} that $\Phi$ is a Gibbs point process for suitable activity $\zeta = \zeta(\beta,p)$. 
    \item Invert the pressure-activity relation: Given $\zeta$, find a $p$ with $\zeta = \zeta(\beta,p)$ and conclude with the uniqueness of Gibbs measures. 
\end{enumerate}
To lighten notation we suppress the $\beta$-dependence from the notation and set $\beta =1$ (equivalently, replace $v$ with $\beta v$). 
For $\Gamma = \{1,\ldots, n\}$ with $n\in \N$ and $\vect z = (z_j)_{j\in \Z}$ define the energy
 \[
	\mathscr H_\Gamma(\vect z) :=  \sum_{ \substack{i,j\in \Z\colon\, i\leq j\\ \{i,\ldots, j\}\cap \Gamma\neq\varnothing}} v(z_i+\cdots + z_{j}). 
\] 
Further define a partition function with boundary conditions $\vect z$ at pressure $p>0$
\[
	Q_{\Gamma;p}^{\vect z} := \int_{\R_+^n} \exp\Bigl( - \mathscr H_{\Gamma}(\vect{z}'_{\Gamma}\vect{z}_{\Z\setminus \Gamma})- p \sum_{i=1}^n z'_i \Bigr)  \dd z'_1\cdots \dd z'_n
\] 
Finally, for measurable $A\subset \R_+^\Gamma$, let 
\[
	\pi_{\Gamma;p}(A\mid \vect z):= \frac{1}{Q_{\Gamma;p}^{\vect z}} 
	\int_{\R_+^n} \1_A({\vect z}'_\Gamma ) \exp\Bigl( - \mathscr H_{\Gamma}(\vect{z}'_{\Gamma}\vect{z}_{\Z\setminus \Gamma})- p \sum_{i=1}^n z'_i \Bigr)  \dd z'_1\cdots \dd z'_n.
\] 
The definitions are easily extended to non-empty finite subsets $\Gamma\subset \Z$. 
In the following 
\[
    \tilde M_n = (Z_{kn-k+1},\ldots, Z_{k n}) \quad (n\in \Z),
\]
see Figure \ref{fig:non-overlapping_M}.

\begin{lem} \label{lem:spacinggibbs}
	Suppose that $(\tilde M_n)_{n\in \Z}$ is a stationary Markov chain with transition kernel $P=P_p= P_{\beta,p}$. 
    Then $(Z_n)_{n\in \Z}$ satisfies the DLR equations at $\beta,p$, i.e., 
	\[
		\P\Bigl( (Z_j)_{j\in \Gamma} \in A\, \Big|\, (Z_j)_{j\in \Z\setminus \Gamma}\Bigr) 
		 = \pi_{\Gamma;p}\bigl( A\mid (Z_j)_{j\in \Z}\bigr)
	\] 
	almost surely, for all finite non-empty subsets $\Gamma \subset \Z$ and all measurable $A\subset \R_+^\Gamma$. 
\end{lem} 

\begin{proof}
	It is enough to check the DLR conditions for a family of sets $\Gamma$ that is cofinal, meaning that every finite set is part of some set of the family, see Remark 1.24 in Georgii \cite{Georgii+2011}. 
    We choose discrete intervals whose length is a multiple of $k$. 
    By translational invariance, we need only consider intervals $\Gamma = \{1,\ldots, Nk\}$, $N\in \N$. 
    Write $P(x,\dd y) = P(x,y) \dd y$ (by some abuse of notation we use the same letter for the density and the kernel). 
    One easily checks that 
	\begin{multline*}
		\P\bigl( (\tilde M_1,\ldots, \tilde M_N) \in A\mid \tilde M_j, j \in \Z\setminus \{1,\ldots, N\}\bigr)  \\
		= \frac{1}{\mathrm{norm.}} \int \1_A(\vect m^{(1)},\ldots, \vect m^{(N)}) \\
			\times P(\tilde M_0,\vect m^{(1)}) P(\vect m^{(1)}, \vect m^{(2)})\cdots P(\vect m^{(N)}, \tilde M_{N+1}) \dd \vect m^{(1)}\cdots \dd \vect m^{(N)}\quad \text{a.s.}
	\end{multline*} 
	where the integral is over $((r_\mathrm{hc},\infty)^k)^{N}$ and the normalisation is given by the same integral but without the indicator $\1_A$. 
    We plug in the explicit formula~\eqref{eq:TransitionKernel} for the kernel $P$, take into account cancellations between the numerator and the denominator, and obtain the required DLR equations. 
\end{proof} 

Next, for $\Gamma = \{1,\ldots, N\}$, let us condition in addition on $Z_1+\cdots + Z_N$. 
Given $L> N r_\mathrm{hc}$, define a measure $\lambda_{N,L}$ on $\R_+^N$ by 
\[
	\lambda_{N,L}(A):= \int_{\R_+^{N-1}} \1_{[0,L]}(z'_1+\cdots + z'_{N-1}) \1_A\Bigl(z'_1,\ldots, z'_{N-1}, L - \sum_{i=1}^{N-1}z'_i\Bigr) \dd z'_1\cdots \dd z'_{N-1}. 
\] 
(think Lebesgue measure on the simplex $z'_1+\cdots + z'_N = L$, all $z'_i \geq 0$). 
Define 
\begin{equation}\label{eq:pihatdef}
	\widehat \pi_{N,L} (A\mid \vect z) := \frac{1}{\mathrm{norm.}} \int_{(r_\mathrm{hc},\infty)^N} 
	\1_A(\vect z') \exp\bigl(- \mathscr H_\Gamma({\vect z}'_\Gamma \vect z_{\Z\setminus \Gamma})\bigr) \lambda_{N,L}(\dd \vect z')
\end{equation}
with a normalisation factor defined by the same integral but without the indicator $\1_A(\vect z')$. 

\begin{lem} \label{lem:restricted}
	For $N\in \N$ and $L \geq N r_\mathrm{hc}$, 
	\[
		\P\bigl( (Z_1,\ldots, Z_N) \in A \mid (Z_j)_{j\in \Z\setminus \{1,\ldots, N\}}, Z_1+\cdots + Z_N\bigr) = \widehat \pi_{N,Z_1+\cdots + Z_N}\bigl( A \mid Z_j, j \in \Z\bigr)
	\] 
	almost surely. 
\end{lem} 

\begin{proof}
Let $h\colon\R_+\to \R$ be a measurable bounded function and $\mathscr C_N = \sigma(Z_j, j \in \Z\setminus \{1,\ldots,N\})$. 
Further let $Y$ be a $\mathscr C_N$-measurable bounded function. 
Then 
\begin{align*}
	\E\bigl[ \1_{\{(Z_1,\ldots, Z_N)\in A\}} h(Z_1+\cdots+ Z_N) Y\bigr] 
	& =  \E\Bigl[ Y \E\bigl[ \1_{\{(Z_1,\ldots, Z_N)\in A\}} h(Z_1+\cdots+ Z_N) \mid \mathscr C_N \bigr]\Bigr]
\end{align*} 
By Lemma~\ref{lem:spacinggibbs} and an elementary change of variables, the conditional expectation on the right side is equal to
\begin{multline*}
	\frac{1}{\mathrm{norm.}}\int_{N r_\mathrm{hc}}^\infty \e^{-p L} h(L) \Biggl\{ \int_{(r_\mathrm{hc},\infty)^N} \1_A(\vect z') \exp\bigl( - \mathscr H_\Gamma(\vect z' \vect Z_{\Z\setminus \Gamma}) \bigr) \lambda_{N,L}(\dd \vect z') \Biggr\} \dd L 
\end{multline*} 
where the normalisation is the same integral but without the indicator $\1_A$. 
The inner integral on the right side is equal to 
\[
	\int_{(r_\mathrm{hc},\infty)^N} \widehat \pi_{N,L}(A\mid (Z_j)_{j\in \Z}\bigr) \exp\bigl( - \mathscr H_\Gamma(\vect z' \vect Z_{\Z\setminus \Gamma}) \bigr) \lambda_{N,L}(\dd \vect z').
\]
Backtracking our computations, we get 
\begin{multline*}
	\E\bigl[ \1_{\{(Z_1,\ldots, Z_N)\in A\}} h(Z_1+\cdots+ Z_N) Y\bigr] \\
	 =  \E\Bigl[ Y \E\bigl[ \widehat \pi_{N,Z_1+\cdots + Z_N}(A\mid (Z_j)_{j\in \Z}\bigr)  h(Z_1+\cdots+ Z_N) \mid \mathscr C_N \bigr]\Bigr] \\
	 = \E\Bigl[ \widehat \pi_{N,Z_1+\cdots + Z_N}(A\mid (Z_j)_{j\in \Z}\bigr)   h(Z_1+\cdots+ Z_N) Y \bigr].
\end{multline*} 
The lemma follows. 
\end{proof} 

Next. we express the kernel $\widehat \pi_{N+1,L}$ with an integral over absolute positions instead of spacings. 
For $n\in \N$, $x_1,\ldots, x_n \in \R$ and $\eta \in \mathbf N$, define
\[
    H(\delta_{x_1}+\cdots + \delta_{x_n} \mid \eta):= 
        \sum_{1\leq i < j \leq n} v(|x_i-x_j|) + \sum_{i=1}^n \int v(|x_i-y|) \eta(\dd y)
\]
if the integrals are absolutely convergent, and $\infty$ otherwise. 
Further set $H(0 \mid \eta):=0$. 

\begin{lem} \label{lem:canonical} 
	Let $N\in \N$, $L> N r_\mathrm{hc}$, and $(z_j)_{j\in \Z}\in \R_+^\Z$. Let $(x_j)_{j\in \Z}$ be the sequence defined by $x_0 =0$ and $x_j - x_{j-1} = z_j$. Define 
	$\Delta_{N,L} = \{(x'_1,\ldots,x'_{N})\colon\ 0 \leq x'_1 <\cdots < x'_N< L\}$. Then, assuming $x_{N+1} = z_1+\cdots + z_{N+1} =L$, we have 
	\begin{multline*}
		\widehat \pi_{N+1,L}(A \mid (z_j)_{j\in \Z} ) = \frac{1}{\mathrm{norm.}} \int_{\Delta_{N,L}} \1_A\bigl(x'_1,x'_2-x_1,\ldots,L-x'_{N}) \\
		\times \exp\Bigl( - H \Bigl( \delta_{x'_1}+\cdots + \delta_{x'_n} \Big| \sum_{j\in \Z\setminus \{1,\ldots, n\}} \delta_{x_j} \Bigr)\Bigr) \dd x'_1\cdots \dd x'_{N}. 
	\end{multline*}
\end{lem} 

\begin{proof}
	In the integral~\eqref{eq:pihatdef} for $\Gamma = \{1,\ldots,N+1\}$, we change variables from $z'_1,\ldots, z'_N$ to $x'_j = z'_1+\cdots + z'_j$, $j=1,\ldots, N$. The energy in the Boltzmann factor in~\eqref{eq:pihatdef} contains all interactions with $x'_i$-particles, and in addition possibly interactions between $x\leq 0$ and $x\geq L$. 
	\begin{align*}
		& \mathscr H_{\{1,\ldots,N+1\}}(\vect z'_{\{1,\ldots, N+1\}} \vect z_{\Z\setminus \{1,\ldots, N+1\}}) \\
		&\quad = \sum_{1\leq i < j \leq n} v(x'_j - x'_i) + \sum_{i=1}^n \sum_{j\in \Z\setminus \{1,\ldots, n\}} v(|x'_i - x_j|) + \sum_{i \leq 0}\sum_{j\geq n+1} v(x_j- x_i) \\
		&\quad = H \Bigl( \delta_{x'_1}+\cdots + \delta_{x'_n} \Big| \sum_{j\in \Z\setminus \{1,\ldots, n\}} \delta_{x_j} \Bigr) + \sum_{i \leq 0}\sum_{j\geq n+1} v(x_j- x_i). 
	\end{align*} 
	We combine the change of variables with this formula, take into account cancellations between the integral in the numerator and the integral in the denominator (normalisation) in~\eqref{eq:pihatdef}, and obtain the lemma. 
\end{proof} 

Finally, we introduce finite-volume canonical Gibbs measures with boundary conditions. 
In the canonical ensemble both the number of particles and the system length are fixed. 
We do not assume that there are particles pinned at the system's ends. 
Let $N\in \N$, $\Lambda = (0,L)$ with $L > (N-1) r_\mathrm{hc}$, and $\eta \in \mathbf N$. 
The restriction of $\eta$ to $\R\setminus \Lambda$ is denoted $\eta_{\R\setminus \Lambda}$ (i.e.,\  $\eta_{\R\setminus \Lambda}(B) = \eta(B\setminus \Lambda)$). 
The canonical partition function is 
\[
    \mathcal Z^\eta_{\Lambda,N}:= 
    \frac{1}{N!}\int_{\Lambda^N} \exp\bigl(- H( \delta_{x_1}+\cdots + \delta_{x_N} \mid \eta_{\R\setminus \Lambda})\bigr) \dd x_1\cdots \dd x_N. 
\]
When the partition function is non-zero, the canonical Gibbs measure is a probability measure on $\mathbf N_\Lambda$, the space of counting measures supported in $\Lambda$, given by 
\begin{multline*}
	\mathbf P_{\Lambda,N}^\eta( A) \\
      := \frac{1}{\mathcal{Z}^\eta_{\Lambda,N}} \, \frac{1}{N!}\int_{\Lambda^N} \1_A(\delta_{x_1}+\cdots + \delta_{x_N}) \exp\bigl(- H( \delta_{x_1}+\cdots + \delta_{x_N} \mid \eta_{\R\setminus \Lambda})\bigr) \dd x_1\cdots \dd x_N. 
\end{multline*}
Whether the canonical partition function vanishes or not depends on the boundary condition $\eta$. 
For empty boundary conditions ($\eta_{\R\setminus \Lambda}=0$) the canonical partition function is non-zero if and only if $L > (N-1) r_\mathrm{hc}$; when $\eta$ contains particles pinned at $0$ at $L$, we need the stronger condition $L> N r_\mathrm{hc}$. 

For $N=0$ we define $\mathcal Z_{\Lambda,0}^\eta:=1$ and $\mathbf P_{\Lambda,0}^\eta:= \delta_0$, the Dirac measure concentrated on the empty configuration. 

\begin{lem} \label{lem:canonical2}
    Let $(Z_n)_{n\in \Z}$ be as in Lemma~\ref{lem:spacinggibbs} and $\Phi$ the associated translationally invariant point process. Then for all non-empty intervals $\Lambda =(0,L)$, $\mathcal Z_{\Lambda,\Phi(\Lambda)}^{\Phi_{\Lambda^\mathrm c}} >0$ almost surely and
	\[
	\P( \Phi_\Lambda \in A\mid \Phi_{\Lambda^\mathrm c}, \Phi(\Lambda))= \mathbf P_{\Lambda,\Phi(\Lambda)}^{\Phi_{\Lambda^\mathrm c}}( A) \quad \text{a.s.}
\] 
\end{lem} 

\begin{proof} 
    Let $F,G\colon\mathbf N\to \R$ be bounded measurable functions   
    and $n\in \N_0$. On the event $\Phi(\Lambda)=n$, the points  and $y$ in $\Phi_{\Lambda^\mathrm c}$ with $x\leq 0 < L \leq y$ that are closest to the boundary of $\Lambda$ must satisfy $y-x > (n+1) r_\mathrm{hc}$, almost surely, and one easily checks $\mathcal Z_{\Lambda,n}^{\Phi_{\mathrm c}}>0$ almost surely.  
    Expected values with respect to $\mathbf P_{\Lambda,N}^{\Phi_{\Lambda^\mathrm c}}$ are denoted
    \[
        \langle F \rangle_{\Lambda,N}^{\Phi_{\Lambda^\mathrm  c}}:= \int_{\mathbf N_\Lambda} F \dd \mathbf P_{\Lambda,N}^{\Phi_{\Lambda^\mathrm  c}}.
    \]
    We check that 
    \begin{equation} \label{eq:dlr-canonical}
	\E\Bigl[ F(\Phi_\Lambda) G(\Phi_{\Lambda^\mathrm c}) \1_{\{ \Phi(\Lambda) =n\}}\Bigr] 
	  = \E \Bigl[ \langle F\rangle_{\Lambda,n}^{\Phi_{\Lambda^\mathrm c}} G(\Phi_{\Lambda^\mathrm c}) \1_{\{ \Phi(\Lambda) =n\}} \Bigr].
    \end{equation}
        For $n=0$ the statement is trivial.
        For $n\geq 1$, 
        we need to bring in the Palm distribution. 
        To that aim we notice, for bounded measurable $f\colon\mathbf N\to \R$, 
    \begin{align*}
    	\E\Bigl[ f(\Phi)\Bigr] & = \E\Bigl[ f(\Phi) \sum_{\substack{x\in \Phi\colon\\ x\leq 0}} \1_{\{\Phi((x,0))=0\}} \Bigr] \\
    	  & = \int_{-\infty}^0 \E\Bigl[ f\bigl(\sum_{j\in \Z}\delta_{x+X_j}\bigr)\1_{\{x+X_1> 0\}} \Bigr] \rho\, \dd x
    \end{align*}
    where $\rho= 1/\E_\pi[Z_1]$ is the intensity of $\Phi$ and $x$ represents the right-most point in  $\Phi$ with $x\leq 0$. 
    Thus
    \begin{align*}
    		\E\Bigl[ F(\Phi_\Lambda) G(\Phi_{\Lambda^\mathrm c}) \1_{\{ \Phi(\Lambda) =n\}}\Bigr] 
     & = \int_{-\infty}^0  \E\Bigl[ Y_{x,n} \1_{\{L \leq  x+ X_{n+1}\}}G\Bigl( \sum_{j\in \Z\setminus \{1,\ldots,n\}} \delta_{x+X_j}\Bigr) 
     \Bigr] \rho\, \dd x
    \end{align*} 
    where 
    \[
    	Y_{x,n} = \1_{\{0 < x+X_1\}} \1_{\{ x+X_n < L \}} F\Bigl( \sum_{j=1}^n \delta_{x+X_j}\Bigr). 
    \]
    We condition on $\sigma (X_j, j \in \Z\setminus \{1,\ldots, n+1\})$ (the points outside $(0,L)$). 
    That $\sigma$-algebra is also generated by $X_{n+1} = Z_1+\cdots + Z_{n+1}$ together with the spacings $Z_j$, $j\leq 0$ and $j \geq n+2$. 
    Revisiting the proof of Lemma~\ref{lem:canonical} (this time $x'_j = x+ z'_1+\cdots+ z'_j$) we find 
    \begin{multline*}
    	\E\bigl[ Y_{x,n}\mid X_j, j \in \Z\setminus \{1,\ldots, n+1\}\bigr] \\
    	 = \frac{1}{\mathrm{norm.}}\, \int_{\Delta_{N,L}} F\Bigl( \sum_{j=1}^n \delta_{x'_j}\Bigr) 
    	 \exp\Bigl( - H\Bigl( \sum_{j=1}^n \delta_{x'_j}\Big| \sum_{i\in \Z\setminus \{1,\ldots, n\}} \delta_{x+X_i}\Bigr) \Bigr) \dd x'_1\cdots \dd x'_n
    \end{multline*}
    with normalisation $\P(0 < x+ X_1, x+ X_n < L\mid X_j, j \in \Z \setminus \{1,\ldots,n+1\})$. 
    Thus, the conditional expectation of $Y_{x,n}$ is $\langle F\rangle_{n,\Lambda}^{\eta}$, with $\eta = \sum_{j\in \Z}\delta_{x+X_j}$, times a function that is measurable with respect to the $\sigma$-algebra that we condition on. 
    We may now backtrack in our computations and obtain Eq.~\eqref{eq:dlr-canonical}, for all $F,G,n$. 
    The lemma follows.
\end{proof} 

\begin{proof}[Proof of Theorem~\ref{thm:gibbs-markov}(a)]
    Lemma~\ref{lem:canonical2} says that the distribution of $\Phi$ is a canonical Gibbs state as defined by  Georgii~\cite{georgii1976canonical}. 
    The set $\mathfrak C = \mathfrak C(\beta)$ of canonical Gibbs measures is a convex set.  
    By Theorem 5.6 \cite{georgii1976canonical}, every extremal canonical Gibbs measure satisfies the DLR equations (equivalent to GNZ) for some activity $\zeta$. 
    For finite-range interactions with a hard core, it is known that for every activity $\zeta$, there is a unique Gibbs measure $\mathbf P^\zeta$. Therefore Remark 5.7 in Georgii \cite{georgii1976canonical} applies and yields that the law of $\Phi$ is a superposition
    \[
    	\P(\Phi\in A) = \int \mathsf P_{\zeta}(A)  \mu(\dd \zeta).
    \] 
    with $\mu$ a probability measure on $(0,\infty)$. 
    The Gibbs meausure $\mathsf P_\zeta$ for finite-range interactions with a hard core is known to be translationally invariant. 
    Suppose by contradiction that the measure $\mu$ is not a Dirac measure on a single activity $\zeta$. 
    Then $\mathscr L_\Phi$ is strictly convex combination of two translationally invariant measures, contradicting the ergodicity of the point process. Therefore, there exists a uniquely defined $\zeta$ such that the law of $\Phi$ is $\mathsf P_\zeta$, in particular it is Gibbs at activity $\zeta$. 
     
    Conversely, let $\zeta>0$ be an activity and $\rho = \rho(\zeta)> 1/r_\mathrm{hc}$ the intensity of the Gibbs point process at activity $\zeta$. 
    Let $g(p) = - \frac 1 k \log \lambda_0(p)$ with $\lambda_0(p)$ the principal eigenvalue of the transfer operator; it is also equal to the Gibbs free energy  
    \[
    	g(p) = - \lim_{n\to \infty} \frac  1 n \log Q_{n}(p)
    \] 
    with $Q_{n}(p)= Q_n(\beta,p)$ as in~\eqref{eq:qnp}.
    It is known that $p\mapsto g(p)$ is a strictly concave function that is continuously differentiable with $\ell(p) :=  \partial_p g(p)$ the expected spacing between consecutive points for the sequence $(Z_j)$ with transition kernel $P_p$ \cite{CassandroOlivieri, Dobrushin73, Dobrushin74, GallavottiMiracleSole}. 
    It is not too difficult to check that $\lim_{p\downarrow 0} \ell(p) = \infty$ and $\lim_{p\uparrow \infty}\ell(p) = r_\mathrm{hc}$. 
    Therefore, there exists a uniquely defined pressure $p>0$ so that $\ell(p) = 1/ \rho(\zeta)$. 
    The point process $\Phi^p$ with this pressure is a Gibbs point process with intensity $\rho(\zeta)$.
    As the map $\zeta\mapsto \rho(\zeta)$ is known to be injective \cite{DobrushinMinlos_ContiuityPressure}, it follows that $\Phi^p$ has law $\mathsf P_\zeta$.
    Geometric ergodicity of the chain of blocks of spacings has been proven in Lemma~\ref{lem:LJgeoErg} and part (a) of Theorem~\ref{thm:gibbs-markov} follows. 
\end{proof}

\subsection{Decay of correlations. Proof of Theorem~\ref{thm:gibbs-markov}(b)} \label{sec:gibbs-corr}
By part (a) of the theorem, we know that the stationary sequence of spacings is such that $\tilde M_n:= (Z_{(k-1)n+1},\ldots, Z_{nk})$, $n\in\Z$, is a stationary Markov chain with transition kernel $P_{\beta,p}$ for some $p>0$. For $k=1$, Theorem~\ref{thm:gibbs-markov} follows right away from Theorem~\ref{thm:NewMain} and the properties of the Markov chain checked in Section~\ref{sec:transfer} (with the minor change that $\R$ is replaced with $(r_\mathrm{hc},\infty)$).

For $k\geq 2$, Theorem~\ref{thm:NewMain} is not applicable right away because we modify the definition of the driving chain: Instead of the overlapping blocks of spacings $M_n=(Z_{n-k+1},\ldots, Z_n)$ from Assumption~\ref{ass:markov}, we look at the non-overlapping blocks $\tilde M_n = (Z_{nk-k+1},\ldots, Z_{kn})$, see Figure \ref{fig:non-overlapping_M}. 
However, the proofs of Theorem~\ref{thm:NewMain} and~\ref{thm:renewal} are easily adapted. 
First notice that, for every interval $I= [t-q,t]\subset (0,\infty)$, 
\[
    \E_\mu \Bigl[ \sum_{\ell=0}^\infty \1_{\{X_\ell \in [t-q,t] \}}\Bigr] 
     = \sum_{n=1}^\infty \E_\mu \Bigl[ \sum_{\ell = (n-1)k+1}^{nk} \1_{\{t-X_\ell \in [0,q] \} } \Bigr].
\]
The points $X_{nk-\ell}$, $\ell = 0,\ldots, k-1$ are reconstructed as $X_{nk-\ell} = X_{nk} - \sum_{j=1}^\ell Z_{nk-\ell}$ where $\sum_{j=1}^0 z_j:=0$. 
Define 
\[
    g((z_1,\ldots,z_k), x):= \sum_{\ell=0}^{k-1} \1_{[0,q]}\Bigl( x +z_k + z_{k-1}+ \cdots + z_\ell\Bigr)
\]
and $\tilde X_n:= X_{kn}$, then 
\[
    \E_\mu \Bigl[ \sum_{\ell=0}^\infty \1_{\{X_\ell \in [t-q,t] \}}\Bigr]  = \sum_{n=1}^\infty \E_\mu\Bigl[  \tilde g(\tilde M_n,t - \tilde X_{n}) \Bigr]
\]
and the increment $\tilde Z_n = \tilde X_n - \tilde X_{n-1} = Z_{nk-k+1}+\cdots + Z_{kn}$ is a function of $\tilde M_n = (Z_{nk-k+1},\ldots, Z_{nk})$. 
This is similar to Eq.~\eqref{eq:umurewrite} for the renewal measure of the Markov random walk, with $k$ instead of $r$. For $\vect z = (z_1,\ldots,z_k) \in (r_\mathrm{hc},\infty)^k$ and Borel sets $A\subset (r_\mathrm{hc},\infty)^k$, $B\subset \R_+$, define
\[
    \tilde K(\vect z, A\times B)
    := \int \1_A(z'_1,\ldots, z'_k) \1_B( z'_1+\cdots+ z'_k) P_{\beta,p}\bigl( \vect z , \dd z'_1\times \cdots \times \dd z'_k).
\]
We may assume without loss of generality that $\tilde K$ satisfies the minorisation condition~\eqref{eq:spread_out_minor} for $r=1$ (otherwise, replace $k$ with $kr$). We have already observed that the Markov random walk $(\tilde M_n, \tilde X_n)$ is spread out. 
The driving chain $(\tilde M_n)$ is geometrically ergodic by Lemma~\ref{lem:LJgeoErg} and Assumption~\ref{MAssu:ExpMoment} holds true by Lemma~\ref{lem:ExpMomIncrLJ}.
The arguments from Section~\ref{sec:genproof} apply and yield 
\[
    \sum_{n=1}^\infty \E_\pi\Bigl[  \tilde g(\tilde M_n,t - \tilde X_{n}) \Bigr]
     = \Or(\e^{- \eps t})
\]
as $t\to \infty$, for some $\eps>0$  that does not depend on the precise interval $[0,q]$. The proof of the exponential decay of two-point functions is then completed by arguments similar to Section~\ref{sec:NewMainProof}.  \qedhere

\begin{figure}
    \centering
    \begin{tikzpicture}[x=0.8\linewidth,y=5cm,>=Stealth]

        \draw[->] (0.1,0.7) -- (0.9,0.7);
    
        \foreach \PP in {-3,-2,-1,1,2,3}
            \draw (0.5+\PP*0.1,0.72) -- (0.5+0.1*\PP,0.68) node[below=2pt] {$X_{\PP}$};
    
        \draw (0.5,0.72) -- (0.5,0.68) node[below=2pt] {$X_{0}=0$};
    
        \foreach \PP in {-2,0,-1,1,2,3}
            \draw[->] (0.405+\PP*0.1,0.77) -- (0.495+\PP*0.1,0.77) node[midway,above] {$Z_{\PP}$};

        \node at (0.15,0.77) {$\cdots$};
        \node at (0.85,0.77) {$\cdots$};

        \draw[decorate, decoration={brace,amplitude=10pt}] (0.5,0.85)--(0.69,0.85) node[xshift=-10pt,above=10pt] {$\tilde M_{1}=(Z_{1},Z_{2})$};
        \draw[decorate, decoration={brace,amplitude=10pt}] (0.3,0.85)--(0.49,0.85) node[xshift=-30pt,above=10pt] {$\tilde M_{0}=(Z_{-1},Z_{0})$};
        
    \end{tikzpicture}
    \caption{An example configuration of the Markov random walk with $(\tilde M_n)_n = \bigl((Z_{nk-k+1},\ldots, Z_{kn})\bigr)_{n}$, where $k=2$. Notice that there is no overlap between $(\tilde M_n)$ with different $n$.} 
    \label{fig:non-overlapping_M}
\end{figure}
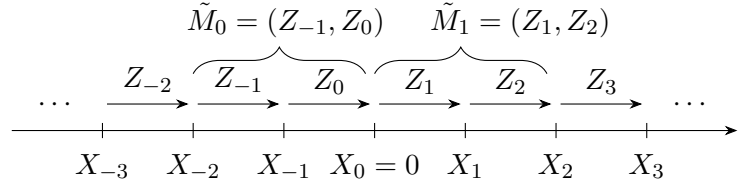

\section{Application to harmonic chains}\label{sec:AR1} 

Here we prove Theorem \ref{thm:harmonic}.
In Section \ref{sec:harmonic(a)} we identify the law of the translation invariant point process $\Phi$ as the limit of the Gibbs measures $\lim_{x_0\to-\infty} \lim_{n\to\infty} \mu_{n,\beta}^{x_0}$ when tested against bounded local functions.
In Section \ref{sec:harmonic(b)} the analysis in Section \ref{sec:StatMech} is adapted to show decay of correlations.

\subsection{Identification of the limit. Proof of Theorem \ref{thm:harmonic}(a)} \label{sec:harmonic(a)}
We first show that the inner limit $\lim_{n \to \infty} \int F \dd\mu_{n, \beta}^{x_0}$ converges to $\mathbb{E}_\zeta [F(\Phi^+ + x_0)]$ using specific mixing-type inequality for $P_\beta$, where $\Phi^+$ consists of points in $\Phi^0$ with non-negative indices, $\Phi^+ = \sum_{i \in \mathbb{N}_0} \delta_{x_i}$, i.e.\ it is the point process corresponding to the one-sided Markov chain $(Z_n)_{n \ge 0}$.
This is done in Proposition \ref{prop:inner_limit}.
The outer limit follows from a result of Berbee \cite[Theorem 6.3.1]{Berbee}.

Let $P_\beta$ be defined as in \eqref{eq:pharmonic}, and $\pi$ its invariant measure.
Let $(Z_n)_{n>0}$ be a Markov chain with transition kernel $P_\beta$, $(X_n)_{n\geq 0}$ be defined as $X_n - X_{n-1} = Z_n$ for $n\geq 1$ and $X_0 = 0$, $\Phi^+ = \sum_{i \geq 0} \delta_{X_i}$. 
Define
\begin{equation*}
    G_\beta(z) = \exp\Bigl( \frac{\beta}{2}(c - \frac{k_1}{2})(z-a)^2\Bigr), \quad \zeta = \mathcal{N}\Bigl(a, \bigl(\beta(c + \frac{k_1}{2})\bigr)^{-1}\Bigr).
\end{equation*}
with $c$ given in~\eqref{eq:cgammadef}. With $\Theta_{x_0}\Phi^+$ we denote the shifted measure $\sum_{i \geq 0} \delta_{x_0+X_i}$.

\begin{prop}
    \label{prop:inner_limit}
    For all $x_0<0$ and bounded, local measurable $F\colon \mathbf{N} \rightarrow \R$, we have 
    \begin{equation*}
        \lim_{n\rightarrow\infty} \int_{\R^n} F(\delta_{x_0} + \dots + \delta_{x_n}) \dd \mu_{n,\beta}^{x_0} = \E_{\zeta}[F(\Theta_{x_0}\Phi^+ )],
    \end{equation*}
    where $\P_{\zeta}$ denote the distribution of the Markov chain $(Z_n)_{n\geq 0}$ with $Z_0 \sim \zeta$.
\end{prop}

To prove this proposition, we need the following two lemmas.
The first lemma rewrites the expectation value with respect to $\mu_{n,\beta}^{x_0}$ as finite-dimensional marginal of the Markov chain $(Z_n)_{n\in \N_0}$ with $\zeta$ as initial distribution.

\begin{lem}
    \label{lem:Gibbs_to_MC}
    For every measurable function $F\colon \mathbf{N} \rightarrow \R$,
    \begin{equation}
        \label{eq:E_zeta}
        \int_{\mathbb{R}^n} F(\delta_{x_0} + \dots + \delta_{x_n}) \, \dd \mu_{n,\beta}^{x_0}(x_1, \dots, x_n) = \frac{1}{\E_{\zeta}[G_\beta(Z_n)]} \E_{\zeta,\beta} \left[ F\left(\sum_{i=0}^n \delta_{x_0+X_i} \right) G_\beta(Z_n) \right]
    \end{equation}
\end{lem}

\begin{proof}
    The energy~\eqref{eq:eharmonic} can be rewritten with the spacings $z_i = x_i - x_{i-1}$ as 
    \[
        \sum_{i=1}^n \frac 12 k_1 (z_i-a)^2 + \sum_{i=1}^{n-1} \frac 12 k_2 (z_i + z_{i+1} - 2 a)^2.
    \]
    Let 
    \begin{equation}    \label{eq:Transfer_Gaussian}    
        K_\beta(z,z') = \exp \Bigl( - \frac \beta 4 k_1(z-a)^2 - \frac \beta 2 k_2(z+z'-2a)^2 - \frac \beta 4 k_1(z'-a)^2\Bigr).
    \end{equation}
    The Boltzmann weight $\exp(- \beta U_n^{x_0})$ is equal to 
    \[
        \e^{- \beta k_1(z_1-a)^2/4}  \Biggl(  \prod_{i=1}^{n-1} K_\beta(z_i,z_{i+1})\Biggr) \, \e^{- \beta k_1(z_n-a)^2/4}. 
    \]
    The integral operator with kernel $K_\beta$ is a variant of the \emph{Kac harmonic operator} see e.g.\ Helffer \cite[Chapter 5.2]{helffer}. It can be diagonalised explicitly. In particular, the principal eigenfunction is $\varphi_\beta(z) = \exp( - \frac 12 \beta c (z-a)^2)$ and the principal eigenvalue is $\lambda(\beta)= (2\pi/(\beta \gamma))^{1/2}$ with $c$ and $\gamma$ given in~\eqref{eq:cgammadef}. An explicit computation yields 
    \[
        \frac{1}{\lambda(\beta)} \frac{1}{\varphi_\beta(z)} K_\beta(z,z') \varphi_\beta(z') = \sqrt{\frac{\beta \gamma}{2\pi}} \exp\Bigl( - \frac 12 \beta \gamma \Bigl(z'-a + \frac{k_2}{\gamma}(z-a)\Bigr)^2  \Bigr) .
    \]
    Let us denote the right side, by some abuse of notation, $P_\beta(z,z')$ (that is, we use the same letter as for the kernel $P_\beta(z,B)$ in~\eqref{eq:pharmonic}). The Boltzmann weight $\exp(- \beta U_n^{x_0})$ becomes 
    \begin{equation*}
        \lambda(\beta)^{n-1} \e^{- \beta k_1 (z_1-a)^2/4} \, \varphi_\beta(z_1) \Biggl(\prod_{i=1}^{n-1} P_\beta(z_i,z_{i+1})\Biggr) \e^{- \beta k_1(z_n-a)^2/4} \, \frac{1}{\varphi_\beta(z_n)}
    \end{equation*}
    which is proportional to 
    \[
        \exp\Bigl( - \frac \beta 2 \Bigl(\frac {k_1}2 + c) (z_1-a)^2 \Bigr)\Bigr)\, \Biggl(\prod_{i=1}^{n-1} P_\beta(z_i,z_{i+1})\Biggr) G_\beta(z_n).
    \]
    and the proof of the lemma is easily concluded. 
\end{proof}

\begin{rem}
    \label{rem:Gibbs_MC_pi} 
    The invariant measure $\pi$ has a probability density function proportional to $\varphi_\beta(z)^2 = \exp( - \beta c(z-a)^2)$, therefore, for  every bounded measurable function $f\colon \mathbb{R}^N \to \mathbb{R}$,
    \begin{equation}
        \label{eq:E_pi}
        \mathbb{E}_\zeta [f(Z_1,\ldots,Z_n)] = \frac{\mathbb{E}_\pi [G_\beta(Z_1) f(Z_1,\ldots,Z_n) ]}{\mathbb{E}_\pi [G_\beta(Z_1)]}.
    \end{equation}
\end{rem}

The second lemma is the aforementioned mixing-type inequality. It exploits that $P_\beta$ has a spectral gap in $L^2(\pi)$ and that $G_\beta$, though unbounded, is square-integrable with respect to $\pi$.

\begin{lem}
    \label{lem:mixing_ineq_G}
    Let $k<n,\,f\colon \mathbb{R}^k \to \mathbb{R}$ be a bounded measurable function, then there exist $C, C' > 0$, such that
    \begin{equation*}
        \bigl|\mathbb{E}_\pi[f(Z_1, \dots, Z_k) G_\beta(Z_1) G_\beta(U_n)] - \mathbb{E}_\pi[f(Z_1, \dots, Z_k) G_\beta(Z_1)]\, \mathbb{E}_\pi[G_\beta(Z_n)] \bigr| \leq C \e^{-C'(n-k)}.
    \end{equation*}
\end{lem}

\begin{proof} 
The function $G_\beta$ is unbounded because $c>k_1/2$, however it is square-integrable with respect to the stationary measure $\pi$ of $P_\beta$. Indeed the stationary measure has probability density proportional to $\varphi_\beta(z)^2 = \exp( - \beta c (z-a)^2)$ with $\varphi_\beta$ the principal eigenfunction of the transfer operator with kernel~\eqref{eq:Transfer_Gaussian}. Thus $G_\beta(z)^2 \varphi_\beta(z)^2= \exp( - \beta k_1 (z-a)^2/4)$ and it follows that $G_\beta\in L^2(\pi)$. 

In addition, $P_\beta$ has a spectral gap in $L^2(\pi)$, i.e., there exists $\delta\in (0,1)$ so that for every $f\in L^2(\pi)$ with $\int f\dd \pi =0$, $|| P f||_{L^2(\pi)} \leq (1-\delta) ||f||_{L^2(\pi)}$. This can be justified in two ways: The analytic argument is that by the Krein-Rutman theorem, the transfer operator has a spectral gap (and it is self-adjoint) in $L^2(\mathrm{Leb})$  with $\mathrm{Leb}$ the Lebesgue measure (this is similar to the proof of Lemma~6.2). The operator $P_\beta$ in $L^2(\pi)$ is unitarily equivalent to the transfer operator and therefore inherits the spectral gap. 

A direct probabilistic argument is that $P_\beta$ is geometrically ergodic because it is autoregressive with Gaussian noise, see Meyn and Tweedie \cite[Chapter 15.2.2]{MeynTweedie}. In addition, one easily checks that $\pi$ is a reversible measure. For reversible chains, geometric ergodicity is equivalent to the existence of a spectral gap in $L^2(\pi)$ \cite[Theorem 1, xxx)]{GLR}  and the statement follows.

To conclude, let $\bar G_\beta = G_\beta - \int G_\beta \dd \pi$. 
By the Markov property, we have
\begin{align*}
&|\mathbb{E}_\pi[f(Z_1, \dots, Z_k) G_\beta(Z_1) G_\beta(Z_n)] - \mathbb{E}_\pi[f(Z_1, \dots, Z_k) G_\beta(Z_1)]\, \mathbb{E}_\pi[G_\beta(Z_n)]| \\
&\quad = \Bigl| \E_\pi \bigl[ f(Z_1, \dots, Z_k) G_\beta(Z_1) \E_{Z_k} [ \bar G_\beta(Z_{n-k}) ] \bigr] \Bigr| \\
&\quad \leq \norm{f}_{\infty} \E_\pi \bigl[ \bigl| G_\beta(Z_1) P_\beta^{n-k} \bar G_\beta(Z_{k})  \bigr|\bigr]\\
&\quad \leq \norm{f}_{\infty} \norm{G_\beta}_{L^2(\pi)}\norm{P_\beta^{n-k} \bar G_\beta}_{L^2(\pi)} \\
&\quad \leq  \norm{f}_{\infty} \norm{G_\beta}_{L^2(\pi)}^ 2 (1-\delta)^{n-k} . \qedhere
\end{align*}
\end{proof}

We are now ready to prove Proposition \ref{prop:inner_limit}.

\begin{proof}[Proof of Proposition \ref{prop:inner_limit}]
Let us write $F_k = F(\sum_{i=0}^k \delta_{x_i})$ as the truncated version of $F$. 
We apply an $\epsilon/3$ argument.  For all $k\leq n$,
\begin{align*}
& \Bigl| \int F_n \, \dd\mu_{n, \beta}^{x_0} - \E_{\zeta} \left[ F ( \Theta_{x_0} \Phi^+  ) \right] \Bigr|
\leq  \underbrace{\Bigl| \int F_n \, \dd\mu_{n, \beta}^{x_0} - \int F_k \, \dd\mu_{n, \beta}^{x_0}\Bigr|}_\textrm{(A)} \\
& \quad + \underbrace{\Biggl|\int F_k \, \dd\mu_{n, \beta}^{x_0} - \E_{\zeta} \left[ F \left( \sum_{i=0}^k \delta_{X_i + x_0} \right) \right]\Biggr| }_\textrm{(B)} 
+ \underbrace{\Biggl|\E_{\beta} \left[ F \left( \sum_{i=0}^k \delta_{X_i + x_0} \right) \right] - \E_{\zeta} \left[ F ( \Theta_{x_0} \Phi^+  ) \right]\Biggr|}_\textrm{(C)}.
\end{align*}
(Remember $\Theta_{x_0} \Phi^+ = \sum_{i\geq 0}\delta_{x_0+ X_i}$.)
We now treat the the terms $(A)$, $(B)$ and $(C)$  separately. To lighten notation we suppress the $\beta$-dependence from the notation. 
\medskip 

\noindent \textbf{Term (B):} By Lemma \ref{lem:Gibbs_to_MC} and Remark \ref{rem:Gibbs_MC_pi}, we have
\begin{align*}
\int F_k \, \dd\mu_{n}^{x_0} &= \frac{\E_{\zeta} \left[ F \left( \sum_{i=0}^k \delta_{X_i + x_0} \right) G(Z_n) \right]}{\E_{\zeta} \left[ G(Z_n) \right]} = \frac{\E_{\pi} \left[ F \left( \sum_{i=0}^k \delta_{X_i + x_0} \right) G(Z_n) G(Z_1) \right]}{\E_{\pi} \left[ G(Z_n) G(Z_1) \right]}.
\end{align*}
By Lemma \ref{lem:mixing_ineq_G} this converges, as $n$ goes to infinity, to
\begin{align*}
\frac{\E_\pi  \left[ F \left( \sum_{i=0}^k \delta_{X_i + x_0} \right) G(Z_1) \right]}{\E_{\pi} \left[ G(Z_1) \right]} = \E_{\zeta} \left[ F \left( \sum_{i=0}^k \delta_{X_i + x_0} \right) \right].
\end{align*}
Therefore, for every fixed $k$, the term (B) goes to zero as $k\to \infty$. 
\medskip 

\noindent \textbf{Term (C):}
As $F$ is local, there exists a bounded interval $\Gamma = [-M,M]$ so that $F$ depends only on the points in $\Gamma$, i.e., $F(\sum_{i\geq 0} \delta_{x_i}) = F(\sum_{i\geq 0} \1_\Gamma(x_i) \delta_{x_i})$. In particular, $F(\sum_{i\geq 0} \delta_{x_i}) = F(\delta_{x_1}+\cdots +\delta_{x_k})$ whenever $\{x_i\colon i \geq k+1\}$ is contained in $\Gamma^\mathrm c$. As a consequence, 
\begin{align} \label{eq:C1}
 \mathrm{(C)}   
     & \leq 2 ||F||_\infty \P_\zeta \Bigl( \exists i \geq k+1\colon X_i +x_0 \in \Gamma\Bigr)
      \leq 2 \norm{F}_\infty \sum_{i=k+1}^\infty \mathbb P_{\zeta}(X_i + x_0 \leq M). 
\end{align}
Each $X_i$ is a Gaussian random variable with mean $ia$. With the Gaussian tail bound $\P(|\xi| \geq z ) \leq 2 \exp(-z^2/2)$ for $\xi\sim\mathcal N(0,1)$ and $z>0$, we get 
\begin{equation} \label{eq:Ctailbound}
    \P_\zeta (X_i + x_0 \leq M) = \P_\zeta( X_i - i a \leq M - x_0 - ia) \leq 2 \exp\Bigl( - \frac{(ia - M + x_0 )^2} {2\mathrm{Var}_\zeta(X_i)}\Bigr)
\end{equation}
for all $i$ that are sufficiently large so that $M-x_0-ia<0$.
The variance of $X_i$ grows linearly with $n$. Indeed, by Remark~\ref{rem:AR1}, we may assume $Z_0\sim \zeta$ and $Z_{n+1} - a= \rho(Z_n-a) + \sigma \xi_n$ with $\rho = - k_2/\gamma \in (-1,0)$, $\sigma^2 = (\beta\gamma)^{-1}$, and $\xi_1,\xi_2,\ldots$ independent standard normal variables also independent from $Z_0$. It follows that $X_n = Z_1+\cdots + Z_n$ satisfies 
\[
    X_n - na  = \frac{1-\rho^{n}}{1-\rho} \rho (Z_0-a) + \sum_{k=1}^n \frac{1-\rho^{n-k+1}}{1-\rho} \, \sigma \xi_k.
\]
Therefore, 
\[
    \mathrm{Var}_\zeta(X_n) = 
     \Bigl(\frac{1-\rho^n}{1-\rho}\, \rho\Bigr)^ 2 \mathrm{Var}_\zeta(Z_0) + \sum_{k=1}^n \Bigl( \frac{1- \rho^{n-k+1}}{1-\rho}\Bigr)^2 \sigma^2.  
\]
On the right side we may bound $(1- \rho^{n-k+1})^2 \leq 4$ (remember $|\rho| < 1$) and we deduce that $\mathrm{Var}_\zeta(X_n) \leq b n$ for some $b>0$ and all $n\in \N$. Turning back to~\eqref{eq:Ctailbound} we find that at fixed $M$ and $x_0$, the probability that $X_i \leq M - x_0$ goes to zero exponentially fast as $i\to \infty$ and then plugging into~\eqref{eq:C1} we find that the term (C) goes to zero as $k \to\infty$. 

\medskip 

\noindent \textbf{Term (A):} The argument is similar to the argument for (C), again we exploit that the underlying measures are Gaussian and we bound variances. To this end, we observe that with $(X_1+x_0,\dots,X_n+x_0) \sim \mu_{n,\beta}^{x_0}$, the vector of spacings $(Z_1,\ldots,Z_n)$ is a random vector with multivariate normal distribution. The mean is $(a,\ldots, a)$ and the covariance matrix $\Sigma_n$ satisfies 
\[
    \frac 12 \langle \vect z', \Sigma_n^{-1} \vect z'\rangle = \sum_{i=1}^n \frac \beta 2 k_1 (z'_i)^2 + \sum_{i=1}^{n-1} \frac \beta 2 k_2 (z'_i + z'_{i+1})^2 \geq \sum_{i=1}^n \frac \beta 2 k_1 (z'_i)^2
\]
for all $\vect z'\in \R^n$, with $\langle \cdot,\cdot \rangle$ the standard scalar product on $\R^n$. It follows that $\Sigma_n^{-1} \geq \beta k_1 I_n$, with $I_n$ the $n\times n$ identity matrix, and then, for all $\vect y \in \R^n$, 
\[
   \langle \vect y, I_n \vect y\rangle =  \langle \Sigma_n^{1/2} \vect y, \Sigma_n^{-1}\Sigma_n^{1/2}  \vect y\rangle \geq \beta k_1 \langle \Sigma_n^{1/2} \vect y, \Sigma_n^{1/2} \vect y\rangle = \beta k_1 \langle \vect y, \Sigma_n \vect y\rangle. 
\]
hence 
\[
    \Sigma_n \leq \frac{1}{\beta k_1} I_n.
\]
Therefore, for every linear function of the spacings, the variance under $\mu_n^{x_0}$ is bounded by the variance it would have for independent normal spacings $Z_i \sim \mathcal N(a,(\beta k_1)^{-1})$. In particular, the variance of $X_i$ under $\mu_n^{x_0}$ is bounded by $i /(\beta k_1)$. Thus, the variance of $X_i$ grows at most linearly with $i$ and the bound is uniform in $n$. 

Armed with this bound we deduce, using Gaussian tail bounds and arguments similar to part (C), that part (A) goes to zero as $k\to \infty$, uniformly in $n$. 
\medskip 

To conclude, given $\eps>0$, first pick $k$ large enough so that (A) and (C) are smaller than $\eps/3$, for all $n$, and then for all $n$ large enough part (B) is smaller than $\eps/3$. This completes the proof of Proposition \ref{prop:inner_limit}.
\end{proof}

We now turn our attention to the outer limit.

\begin{prop}
    \label{prop:outer_limit}
    For any bounded local measurable function $F\colon \mathbf{N} \to \mathbb{R}$,
    \begin{equation*}
        \lim_{x_0\to-\infty} \mathbb{E}_{\zeta, \beta}[F(\Theta_{x_0}\Phi^+)] = \mathbb{E}_{\beta}[F(\Phi)].
    \end{equation*}
\end{prop}

\noindent 
The proof is based on \cite[Theorem 6.3.1]{Berbee}.
\begin{proof}
    Let us fix $\beta>0$ and drop it from our notation.
    Let $\Gamma$ be a bounded interval so that $F$ depends only on the points in $\Gamma$, see term $(C)$ in the proof of \ref{prop:inner_limit}.
    Without loss of generality, we may assume that $\Gamma$ is a subset of $\R_+$.
    Indeed, any local function $F$ can be shifted such that this is true: there exists a finite $s_F \geq 0$ such that $F\circ \Theta_{-s_F}$ depends only on points in $\{x+s_F \mid x \in \Gamma\} \subset \R_+$.
    This finite shift does not affect the limit in the left-hand-side and $\Phi$ in the right-hand-side is translation invariant.
    Hence it suffices to look at local functions with support in $\R_+$.

    Denote $\mathbf{N}_+$ the space of locally finite counting measures on $\R_+$ with a corresponding $\sigma$-algebra $\mathscr{M}_+ = \sigma(\bigl\{\{ \eta\in\mathbf{N}_+\, \colon \eta(B)=k \} \vert B\in \B_{\R_+}, k\in \N_0\bigr\})$.
    We observe that the Markov chain $(Z_n)_{n\geq 0}$ is ergodic, has strictly positive drift and is spread out, see Section \ref{sec:harmonic(b)}.
    Moreover, this implies that it is also aperiodic and positive Harris recurrent \cite[Chapter 13]{MeynTweedie}.
    Other conditions needed to apply \cite[Theorem 6.3.1]{Berbee}, i.e.\ $(Z_n)_{n\in\Z}$ is weakly Bernoulli and satisfies \cite[(5.1.3)]{Berbee}, are verified in the proof of \cite[Theorem 6.4.5]{Berbee} for aperiodic, positive Harris chains.
    Hence, we can apply \cite[Theorem 6.3.1]{Berbee} and it follows that for $\pi$-a.a.\ $m$,
    \begin{equation*}
        \lim_{t\to-\infty} \sup_{A\in\mathscr M_+} \bigl\vert\P_m(\Theta_{t}\Phi^+\vert_{\R_+} \in A ) - \P(\Phi\vert_{\R_+} \in A ) \bigr\vert =0,
    \end{equation*}
    where $\eta\vert_{\R_+} := \eta(\cdot \cap \R_+)$.
    Equivalently, we have for $\pi$-a.a.\ $m$,
    \begin{equation*}
        \lim_{t\to-\infty}\sup_{\substack{G\colon \mathbf{N}_+ \to \R\, \text{meas.}\\ \norm{G}_{\infty} \leq 1}} \Bigl| \E_m[G(\Theta_{t}\Phi^+\vert_{\R_+})] - \E[G(\Phi\vert_{\R_+})] \Bigr| = 0.
    \end{equation*}
    For $F$ local with support $\Gamma \subset \R_+$, we have $F(\eta) = F(\eta\vert_{\R_+})$.
    We may take $G\colon\mathbf{N}_+\to \R, G(\eta)=F(\eta)/\norm{F}_\infty$, where measurability of $G$ with respect to $\mathscr M_+$ follows from locality of $F$.

    The proposition now follows from the dominated convergence theorem (or uniform integrability of $\E_m[F(\Theta_{t}\Phi^+)]$).
\end{proof}

\subsection{Decay of correlations. Proof of Theorem \ref{thm:harmonic}(b)} \label{sec:harmonic(b)} 
To prove Theorem \ref{thm:harmonic}(b), we note that by Remark \ref{rem:AR1}, the corresponding Markov chain is an autoregressive process with Gaussian noise.
It follows that it is Harris recurrent and geometrically ergodic, see Meyn and Tweedie \cite[Chapter 15.5]{MeynTweedie}.
Examining \eqref{eq:pharmonic} reveals that it is absolutely continuous with respect to the Lebesgue measure, and it has positive drift $a$.
Thus, to apply Theorem \ref{thm:NewMain}, we only need to check the Assumption \ref{MAssu:ExpMoment}.

To this end, we adapt the proof of Lemma \ref{lem:ExpMomIncrLJ}.
Recall the transfer operator $\mathcal{K}_\beta$ on $L^2(\R)$ with integral kernel \eqref{eq:Transfer_Gaussian}.
It is easy to see that this operator is (strongly) positive and Hilbert-Schmidt.

Although there is no pressure term, the analysis is almost the same: writing $\varphi_{0;\beta}$ and $\lambda_0(\beta)$ for the principal eigenfunction and eigenvalue, respectively, we have
\begin{align}
    &\E_\pi[\e^{\theta |X_{n}|}] = \int  \e^{\theta |\sum_{i=1}^n z_i|} \varphi_{0;\beta}(z_1) \left( \prod_{i=1}^{n-1} \frac{1}{\lambda_0(\beta)} K_\beta (z_i, z_{i+1}) \right) \varphi_{0;\beta}(z_n) \dd z_1 \dd z_2\cdots \dd z_n \nonumber \\
    & \quad \leq \int  \e^{\theta \sum_{i=1}^n |z_i| + \theta\sum_{i=2}^{n-1} |z_i|} \varphi_{0;\beta}(z_1) \left( \prod_{i=1}^{n-1} \frac{1}{\lambda_0(\beta)} K_\beta (z_i, z_{i+1}) \right) \varphi_{0;\beta}(z_n) \dd z_1 \dd z_2\cdots \dd z_n \nonumber \\
    & \quad = \int \varphi_{0;\beta}(z_1) \left( \prod_{i=1}^{n-1} \frac{1}{\lambda_0(\beta)} K_{\beta;\theta} (z_i, z_{i+1}) \right) \varphi_{0;\beta}(z_n) \dd z_1 \dd z_2 \cdots \dd z_n,
\end{align}
where $K_{\beta;\theta} (x,y)= K_\beta(x,y) + \exp(\theta(|x|+|y|))$.
For sufficiently small $\theta>0$, the operator $K_\theta$ is strongly positive and Hilbert-Schmidt, hence we can conclude, as in the last section, that Assumption~\ref{MAssu:ExpMoment} holds.

Alternatively, we could use the autoregressive nature of the process.
Denote $\rho_{\text{AR}} = -\frac{k_2}{\gamma}$ and $\tilde \xi_k = \frac{1}{\sqrt{\beta \gamma}} \xi_k$.
By recursively applying the autoregressive relation \eqref{eq:AR1}, we have the identity
\begin{equation}
    X_n = \frac{1-\rho_{\text{AR}}^{n}}{1-\rho_{\text{AR}}} \rho_{\text{AR}} (Z_0-a) + \sum_{k=1}^{n} \frac{1-\rho_{\text{AR}}^{n-k+1}}{1-\rho_{\text{AR}}} \tilde\xi_k + na. 
\end{equation}
Thus,
\begin{align*}
    &\E_\pi \bigl[\e^{\epsilon \abs{X_n}}\bigr] = \E_\pi \left[ \exp\left(\epsilon \left| \frac{1-\rho_{\text{AR}}^{n}}{1-\rho_{\text{AR}}} \rho_{\text{AR}} (Z_0-a) + \sum_{k=1}^n \frac{1-\rho_{\text{AR}}^{n-k+1}}{1-\rho_{\text{AR}}}\tilde\xi_k + na \right|\right) \right] \\
    & \qquad\leq \E_\pi \left[ \exp\left(\epsilon \left| \frac{1-\rho_{\text{AR}}^{n}}{1-\rho_{\text{AR}}}\right| |Z_0-a| + \sum_{k=1}^n \left|\frac{1-\rho_{\text{AR}}^{n-k+1}}{1-\rho_{\text{AR}}}\right| |\tilde\xi_k| + na\right) \right] \\
    & \qquad\leq \e^{\epsilon na}\E_\pi \left[\exp\left(\frac{\epsilon}{1-|\rho_{\text{AR}}|}\abs{Z_0-a}\right)\right]\E_\pi \left[\exp\left(\frac{\epsilon}{1-|\rho_{\text{AR}}|}\abs{\tilde\xi_1}\right)\right]^n,
\end{align*}
where we have used
\begin{equation*}
    \left|\sum_{i=0}^{n-1} \rho_{\text{AR}}^i\right| \leq \sum_{i=0}^{n-1} |\rho_{\text{AR}}|^i \leq \frac{1}{1-|\rho_{\text{AR}}|}.
\end{equation*}
We can conclude by the fact that both $\abs{Z_0-a}$ and $\abs{\tilde\xi_1}$ are distributed with folded normal distribution (i.e., the law of the absolute value of a normal variable, see \cite{FoldedNormal1961}), and thus their moment generating function exists in a neighbourhood around the origin. 
\qed

\bigskip
\subsection*{Acknowledgements.} 
This work is supported by the Deutsche Forschungsgemeinschaft (DFG, German Research Foundation) – TRR 352 – project number 470903074. S.J.\ thanks Gerold Alsmeyer, Sascha Kissel and Elena Pulvirenti for helpful discussions on renewal theory, canonical Gibbs measures, and geometric ergodicity for Markov chains defined through transfer operators.

\bibliography{references.bib}{}
\bibliographystyle{amsplain}

\end{document}